\documentclass[11pt]{article} 
\usepackage[margin=1in]{geometry}
\usepackage{setspace}
\usepackage[autostyle,english=american]{csquotes}

\usepackage{amsthm}
\usepackage{amsfonts}
\usepackage{newtxmath}
\usepackage{amsmath}

\newtheorem{example}{Example}
\newtheorem{definition}{Definition}
\newtheorem{proposition}{Proposition}
\newtheorem{theorem}{Theorem}
\newtheorem{lemma}{Lemma}

\usepackage{url}

\usepackage{endnotes}
\let\footnote=\endnote

\usepackage{natbib}
 \bibpunct[, ]{(}{)}{,}{a}{}{,}%
\usepackage{subfig}
\usepackage{caption}
\usepackage{tikz}
\usetikzlibrary{shapes.geometric}
\usepackage[ruled,vlined,linesnumbered]{algorithm2e}
\usepackage{mathtools}
\usepackage{enumitem}
\usepackage{comment}
\usepackage{graphicx}
 \usepackage{tabularx}
 \usepackage{multirow}
\usepackage{booktabs}

\newcommand{\polyopt}{MLP}
\newcommand{\recmcclin}{RML}

\newcommand{\R}{\mathbb{R}}

\newcommand{\B}{\mathbb{B}}

\newcommand{\bfx}{\boldsymbol{x}}

\newcommand{\bfu}{\boldsymbol{u}}
\newcommand{\bfv}{\boldsymbol{v}}

\newcommand{\bfy}{\boldsymbol{y}}

\newcommand{\bfsetV}{\boldsymbol{\cal V}}

\newcommand{\bflambda}{\boldsymbol{\lambda}}
\newcommand{\bfmu}{\boldsymbol{\mu}}

\newcommand{\vhat}{\widehat{v}}

\newcommand{\bfvhat}{\widehat{\bfv}}

\newcommand{\lambdahat}{\widehat{\lambda}}
\newcommand{\muhat}{\widehat{\mu}}

\newcommand{\bflambdahat}{\widehat{\boldsymbol{\lambda}}}
\newcommand{\bfmuhat}{\widehat{\boldsymbol{\mu}}}

\newcommand{\lambdatilde}{\widetilde{\lambda}}
\newcommand{\mutilde}{\widetilde{\mu}}

\newcommand{\bflambdatilde}{\widetilde{\boldsymbol{\lambda}}}
\newcommand{\bfmutilde}{\widetilde{\boldsymbol{\mu}}}

\newcommand{\optlambda}[1]{\lambda^{#1}}
\newcommand{\optmu}[1]{\mu^{#1}}

\newcommand{\optbflambda}[1]{\bflambda^{#1}}
\newcommand{\optbfmu}[1]{\bfmu^{#1}}

\newcommand{\indexset}{J}

\newcommand{\graph}{{\cal G}}
\newcommand{\graphnodes}{{\cal N}}
\newcommand{\grapharcs}{{\cal A}}
\newcommand{\graphnode}{\indexset}
\newcommand{\grapharc}{\mathsf{a}}
\newcommand{\pairarc}{\mathsf{pair}}
\newcommand{\tripleSet}{{\cal T}}
\newcommand{\triple}{\mathsf{t}}
\newcommand{\head}{\mathsf{head}}
\newcommand{\tails}{\mathsf{tails}}
\newcommand{\tailone}{\mathsf{tail1}}
\newcommand{\tailtwo}{\mathsf{tail2}}

\newcommand{\mccormick}{{\mathcal{E}}}

\newcommand{\set}[1]{\left\{#1\right\}}

\newcommand{\nmonomials}{m}
\newcommand{\nvariables}{n}

\begin{document}



\title{Recursive McCormick Linearization of Multilinear Programs}

\author{Arvind U Raghunathan \\
Mitsibushi Electric Research Laboratories, Cambridge, MA 02139
\and
Carlos Cardonha \\
University of Connecticut, Storrs, CT 06269 
\and
David Bergman \\
University of Connecticut, Storrs, CT 06269 \\
\and
Carlos J Nohra \\
Amadeus, Irving, TX 75062
} 
\date{}
\maketitle

\abstract{Linear programming (LP) relaxations are widely employed in exact solution methods for multilinear programs (\polyopt).  One example is the family of Recursive McCormick Linearization (\recmcclin) strategies, where bilinear products are substituted for artificial variables, which deliver a relaxation of the original problem when introduced together with  concave and convex envelopes. In this article, we introduce the first systematic  approach for identifying~\recmcclin{}s, in which we focus on the identification of linear relaxation with a small  number of artificial variables and with strong LP bounds. We present a novel mechanism for representing all the possible \recmcclin{}s, which we use to design an exact  mixed-integer programming (MIP) formulation for the identification of minimum-size \recmcclin{}s; we show that this problem is NP-hard in general, whereas a special case is  fixed-parameter tractable. Moreover, we explore structural properties of our formulation to derive an exact MIP model that identifies~\recmcclin{}s of a given size with the best possible relaxation bound is optimal. Our numerical results on a collection of benchmarks indicate that our algorithms outperform the~\recmcclin{} strategy implemented in state-of-the-art global optimization solvers.}%

\section{Introduction}

This article introduces new techniques for linearizing multilinear terms in optimization problems. We focus on  unconstrained \emph{multilinear programs} (\polyopt)  defined over~$\Omega = [0,1]^n$ or~$\Omega = \{0,1\}^n$, where~$n$ is the number of variables.
An~\polyopt{} is formulated as  
\begin{equation}
\begin{aligned}
    \min\limits_{\bfx \in \Omega}    &\; f(\bfx) 
    = \sum\limits_{i=1}^\nmonomials \alpha_i 
    \prod_{j \in \indexset_i} x_j.
\\
\end{aligned} \label{polyopt}
\end{equation}
We use~$\bfx = (x_1, \ldots, x_\nvariables)$ to denote a vector in~$\Omega$. Function~$f(\bfx)$ consists of~$\nmonomials$ monomials.  Each monomial~$i \in [\nmonomials]$ is composed of  a  coefficient $\alpha_i \in \mathbb{R}$ and a term~$f_i(\bfx) \coloneqq \prod\limits_{j \in \indexset_i} x_j$, i.e., $f_i(\bfx)$ is the product of the variables whose indices are given by a subset~$\indexset_i$ of~$[\nvariables]$. 

\begin{example}\label{example}  Consider the~\polyopt{} $\min\limits_{\bfx \in [0,1]^4}  f(\bfx) = x_1 x_2 x_3 - x_2 x_3 x_4 - x_1 x_3 x_4$.
 This~\polyopt{} consists of~$m = 3$ monomials, which are defined over
 $n = 4$  variables with domain~$[0,1]$. The first monomial of~$f(\bfx)$, represented by~$\alpha_1 f_1(\bfx) = x_1 x_2 x_3$, is described by the coefficient~$\alpha_1 = 1$ and $\indexset_1 = \set{1,2, 3}$, which gives the term~$f_1(\bfx) = x_1 x_2 x_3$.
\end{example}

Exact methods for solving nonlinear programs
typically rely on the derivation of  relaxations of~$f(\bfx)$ (see e.g., \citep{burer2012milp}). 
A popular approach pioneered by \citet{McCormick1976} consists of obtaining a convex relaxation for each monomial~$\alpha_if_i(\bfx)$. This approach can be used in the construction of  a linear programming relaxation of an~\polyopt{} and is employed in state-of-the-art global optimization solvers, such as \texttt{BARON}~(\cite{Sahinidis1996}), \texttt{Couenne}~(\cite{BeLeLiMaWa08}); see also e.g., \citet{FloudasViswa93} and~\citet{SmithPantelides99}. The main idea is to sequentially replace each bilinear product in the~\polyopt{} by an auxiliary variable, which is 
connected with the components of the bilinear product 
through channeling constraints, such as McCormick inequalities (\cite{McCormick1976}), to yield a relaxation of the original problem. By iteratively applying such operations, one can linearize and solve the problem by branch and bound. We refer to a linearization strategy following the iterative procedure described above as a \emph{Recursive McCormick Linearization} (\recmcclin).  
The number of auxiliary variables and the quality of the linear programming (LP) relaxation bound of a linearized model vary across different \recmcclin{}s. We illustrate these differences in Example~\ref{example2}.

\begin{example}\label{example2}
Figures~\ref{ex:rml1} and~\ref{ex:rml2}
depict two~\recmcclin{}s for the~\polyopt{} shown in Example~\ref{example}.
 \begin{figure}[htp]
\begin{minipage}{.47\linewidth}
\centering
\begin{eqnarray*}
    \underbrace{x_1 x_2}x_3 
        &\quad 
    \underbrace{x_2 x_3} x_4 
        &\quad 
    \underbrace{x_1 x_3} x_4 \\
    \underbrace{y_{\{1,2\}} x_3}
        &\quad 
    \underbrace{y_{\{2,3\}} x_4}
        &\quad 
    \underbrace{y_{\{1,3\}}x_4} \\
    y_{\{1,2,3\}}
        &\quad 
    y_{\{2,3,4\}}
        &\quad 
    y_{\{1,3,4\}}
\end{eqnarray*}
\caption{\recmcclin{} with 10 variables and LP bound~$\frac{-4}{3}$.}
\label{ex:rml1}
\end{minipage}%
  \hfill
\begin{minipage}{.47\linewidth}
\centering
\begin{eqnarray*}
    x_2 \underbrace{x_1 x_3}
        &\quad 
    x_2 \underbrace{x_3 x_4}
        &\quad 
    x_1 \underbrace{x_3 x_4} \\
    \underbrace{x_2 y_{\{1,3\}}}
        &\quad 
    \underbrace{x_2 y_{\{3,4\}}}
        &\quad 
    \underbrace{x_1 y_{\{3,4\}}} \\
    y_{\{1,2,3\}}
        &\quad 
    y_{\{2,3,4\}}
        &\quad 
    y_{\{1,3,4\}}
    \end{eqnarray*}
\caption{\recmcclin{} with 9 variables and LP bound~$-1$.}
\label{ex:rml2}
\end{minipage}
\end{figure} 
The~\recmcclin{} in Figure~\ref{ex:rml1} uses ten variables in total (from which six are artificial variables) and delivers an LP bound of~$\frac{-4}{3}$, whereas the~\recmcclin{} in Figure~\ref{ex:rml2} uses five artificial variables and has an LP bound of~$-1$. 
 \end{example}

To the best of authors' knowledge, this article is the first systematic study of~\recmcclin{}s for~\polyopt{}s focused on the number of introduced variables and the relaxation bound of the entire~\polyopt{}, rather than just one of its monomials. In particular, we show exact approaches for the identification of~\recmcclin{}s that have
 (1) a small number of auxiliary variables; or (2)  a tight LP relaxation bound.  Branch-and-bound algorithms benefit from (1) because fewer search nodes need to be explored  and from (2) because tigher relaxations typically result in more pruning  during the solution process. Therefore, in Example~\ref{example2}, the~\recmcclin{} in Figure~\ref{ex:rml2} is preferred over the~\recmcclin{} in Figure~\ref{ex:rml1}.

 Our main contributions can be summarized as follows: 
 \begin{itemize}
    \item \textbf{Problem Definition}: We formalize the study of~\recmcclin{}s as optimization problems with respect to size (number of auxiliary variables introduced by the linearization strategy) and LP bound;
    \item \textbf{Minimum-Size~\recmcclin{}:} We present numerous results about the identification of minimum-size~\recmcclin{}s. We show that the problem is NP-hard even if all monomials have degree at most three; we also present a fixed-parameter tractable algorithm for this special case of the problem. Furthermore, we present a greedy algorithm, which can deliver arbitrarily poor results but typically performs well in practice. Finally, we propose an exact MIP model for finding minimum-size~\recmcclin{}s.
     \item \textbf{Best-Bound \recmcclin{}}: We present an exact MIP model for finding best-bound \recmcclin{}s of any given size. Our results rely on the transformation of a two-level MIP formulation into a single-level MIP based on bounds we derive for dual variables of the inner-level sub-problem.
    \item \textbf{Numerical study}: We compare the performance of our algorithms with the linearization strategies adopted in practice  using benchmark instances that have been traditionally adopted by the global optimization community.
 \end{itemize}

The remainder of this paper is organized as follows. \S\ref{sec:literature} provides an overview of the literature. \S\ref{sec:description} formalizes the problem and introduces the notation used in the paper.
\S\ref{sec:minimumlinearization} and~\S\ref{sec:BB} present our results involving minimum-size~\recmcclin{}s and best-bound~\recmcclin{}s, respectively. \S\ref{sec:computational_results} presents our numerical studies. Finally, \S\ref{sec:Conclusion} concludes the article and discusses directions for future work.

\section{Literature Review}\label{sec:literature}

Multilinear functions appear in a variety of nonconvex optimization problems~\citep{HorstTuy}. In addition, multilinear functions arise when the Reformulation-Linearization Technique~\citep{SheraliAdamsBook} is used to approximate the convex
hull of general classes of mathematical programs, including polynomial optimization problems.  A recent survey by \cite{AhmadiMajumdar2016} presents a number of applications that can be modeled as polynomial optimization problems. 

The construction of convex lower bounding and concave upper bounding functions is key to global optimization of an~\polyopt. A standard approach to solving an~\polyopt{} is to recast~\eqref{polyopt} as
\begin{equation}
\begin{aligned}
    \min\limits_{\bfx \in [0,1]^n, y \in [0,1]^{m} } &\; \sum\limits_{i=1}^m \alpha_i y_{\indexset_i} \\
    \text{s.t.} &\; y_{\indexset_i} = f_i(\bfx) \,\forall\, i = 1,\ldots,m.
\end{aligned} \label{epolyopt}
\end{equation}
The feasible region defined by the nonlinear equalities in~\eqref{epolyopt} are approximated by linear inequalities, in a process that has been termed in the literature as \emph{linearization}.  A popular approach to linearize the nonconvex region defined by $y_{\indexset_i} = f_i(\bfx)$ when the variables are in $\{0,1\}^n$ is to replace it with its convex hull \citep{Glover1974}.  For the case where variables are binary and continuous, the RML  procedure described in the introduction is used to obtain a linearization.  It is known that the McCormick inequalities~\citep{McCormick1976} define the convex hull for a single term when the variables are in $[0,1]^n$~\citep{RyooSahinidis01} or when the bounds are symmetric around zero~\citep{LuedtkeNamzifar2012}.  Global optimization solvers such as \texttt{BARON}~\citep{Sahinidis1996},  \texttt{Couenne}~\citep{BeLeLiMaWa08,couenne}  and other approaches~(\cite{FloudasViswa93,SmithPantelides99}) solve the \polyopt{} by constructing an LP relaxation using a \recmcclin.  However, such a relaxation  is known to be weak for an~\polyopt{}~\citep{LuedtkeNamzifar2012} and can be strictly contained inside the convex hull of the feasible region of~\eqref{epolyopt}.

An explicit characterization of the convex hull of~\eqref{epolyopt} is known to be polyhedral~\citep{Crama1993,Rikun1997,Sherali1997,Floudas00,baronbook,Tawarmalani10}.  However, it is computationally prohibitive to directly incorporate the convex hull characterization in the LP relaxations since the size of the formulation is exponential in $n$. Hence, it is desirable to find a relaxation that combines the strengths of the \recmcclin{}-based LP relaxation and the convex hull-based LP relaxation. A number of articles~\citep{Bao2009,DelPia2018,DelPia2021} derive cutting planes to strengthen the LP relaxation obtained from an~\recmcclin.    
\cite{dpks:20} report improved computational performance from using the cuts identified in~\cite{DelPia2021}  at the root node LP relaxation obtained from full sequential \recmcclin. 

The preceding discussion clearly demonstrates the fundamental role played by \recmcclin{} in the global optimization of \polyopt.  Interestingly, as shown in Example~\ref{example}, a given MLP can yield a wide range of \recmcclin{}s, i.e. linearizations are not necessarily unique. Missing from the literature is a systematic study of how different \recmcclin{}s can be obtained and, more importantly, how one can construct the smallest possible linearization, in terms of the number of introduced variables.  Note that we need $|\indexset_i| -1$ auxiliary variables to linearize a given monomial~$f_i(\bfx)$. However, when considering a polynomial with several terms, a judicious choice of linearization can lead to a significant reduction in the number of auxiliary variables by exploiting commonality in the bilinear terms among the monomials. Unfortunately, a greedy approach does not necessarily yield the best results (see e.g., Example~\ref{example}), so the identification of a minimum-size~\recmcclin{} relies on more sophisticated strategies.

Another aspect that has not been explored is the question of identifying best-bound \recmcclin{}s, i.e.,  \recmcclin{}s that yields the best relaxation bound when the  number of auxiliary variables introduced by the linearization is constrained. 
In a related line of work, \citet{Cafieri2010} and \citet{BelottiCafieri2013} consider different ways of computing convex hulls of a quadrilinear term by exploiting associativity; in particular, they prove that having fewer groupings of longer terms yields tighter convex relaxations. The work of~\cite{SpeakmanLee17}, \cite{SpeakmanYu2017}, \cite{LeeSkipper2018}, and~\cite{SpeakmanAverkov22} study the polyhedral relaxations by comparing the volumes of the resulting relaxations, but do not consider the identification of  best-bound \recmcclin{}s.

%

\section{Linearization of Multilinear Programs}
\label{sec:description}

The typical algorithm for solving an \polyopt, which is commonly employed in solvers, is to sequentially reduce the number of variables in each multilinear term.
Consider any index $i \in [\nmonomials]$  and the corresponding term $f_i(\bfx) = \prod\limits_{j \in \indexset_i} x_j$.
One can reduce the number of variables in this expression through an iterative introduction of artficial variables. First, select any two indices $j_1,j_2 
\in \indexset_i$. 
Then, introduce a variable $y_{\set{j_1, j_2}}$ that corresponds to the bilinear product $x_{j_1} x_{j_2}$ and rewrite~$f_i(\bfx)$ as
\[
f_i(\bfx) = 
y_{\set{j_1, j_2}}  \prod\limits_{j \in \indexset_i \backslash \set{j_1,j_2}} x_j.
\]
The equality above and, in particular, directly expressing $y_{\set{j_1, j_2}} = x_{j_1}  x_{j_2}$ does not eliminate nonlinearity, but we can use McCormick convex and concave envelopes to relax this expression (\cite{McCormick1976}):
\begin{subequations}\label{eq:McCormickEnvelopes}
\begin{align}
     y_{\set{j_1, j_2}} & \geq 0 \\
     y_{\set{j_1, j_2}} &- x_{j_1} - x_{j_2} + 1 \geq 0 \\
     y_{\set{j_1, j_2}} &- x_{j_1}  \leq 0 \\
     y_{\set{j_1, j_2}} &- x_{j_2} \leq 0 
\end{align}
\label{envelopes}
\end{subequations}
 We denote the McCormick inequality system that linearizes the bilinear product~$x_{j_1}  x_{j_2}$ by introducing an auxiliary variable $y_{\set{j_1, j_2}}$ and the convex and concave envelopes in~\eqref{envelopes} as $\mccormick\left(\triple\right)$ with $\triple = \left(x_{j_1},x_{j_2},y_{\set{j_1, j_2}}\right)$.
 This procedure can be recursively applied to the remaining bilinear products of original and artificial variables until~$f_i(\bfx)$ is completely linearized.
 To simplify the notation, we refer to the variable $x_j$ also as $y_{\set{j}}$.  Therefore, the variables in our models are given by~$y_{\indexset}$, where~$\indexset$ is an index set $\indexset \subseteq [\nvariables]$.  


\subsection{Recursive McCormick Relaxation~(\recmcclin{})}

  For any $i \in [\nmonomials]$, let $\graphnodes_i \coloneqq \set{\indexset: \indexset \subseteq \indexset_i, \indexset \neq \emptyset}$ be the family of non-empty subsets of indices of the variables in monomial $i$, and let $\graphnodes = \bigcup\limits_{i \in [\nmonomials]} \graphnodes_i$.  For any~$\indexset''$ in~$\graphnodes$ such that~$|\indexset''| \geq 2$, a \textit{triple}~$\triple = (\indexset,\indexset',\indexset'')$ describes a partition of~$\indexset''$ into two non-empty sets~$\indexset$ and~$\indexset'$ such that 
  $\indexset \cap \indexset' = \emptyset$ and $\indexset \cup \indexset' = \indexset''$. 
  We assume that 
the first two elements of any triple are arranged in lexicographical order.  In this way, we can uniquely define $\tailone(\triple)$, $\tailtwo(\triple)$, and~$\head(\triple)$ as the first, second, and third element of $\triple$, respectively. Finally, let   $\tripleSet_i \coloneqq \{\triple: \head(\triple) \in \graphnodes_i \}$ and $\tripleSet = \bigcup\limits_{i \in [\nmonomials]} \tripleSet_i$ be the set of all possible triples associated with~$\graphnodes_i$ and~$\graphnodes$, respectively, and let~$\tails(\triple) \coloneqq \{\tailone(\triple),\tailtwo(\triple)\}$.

\begin{definition}
\label{def:rmr}
A \emph{Proper Triple Set}
for an~\polyopt{} is a set of triples $T \subseteq \tripleSet$ for which there exists a subset $T' \subseteq T$ satisfying the following conditions:
\begin{enumerate}[align = left,label = \textbf{RMP \arabic*}]
    \item \label{rmpDef1} Every set $\indexset_i$ with $|\indexset_i| > 1$  is the third element of a triple $\triple \in T'$; and
    \item \label{rmpDef2} If a set $\indexset$ such that $|\indexset| > 1$ is the first or second element of a triple $\triple \in T'$, then $\indexset$ is the third element of a different triple in $T'$.
\end{enumerate}
\end{definition}
A proper triple set~$T$ defines a \textit{Recursive McCormick Relaxation} (\recmcclin{}) of an~\polyopt{}  over the set of variables~$y_{\indexset}$ for each~$\indexset$ in~$\{\head(\triple) : \triple \in T\}$ and subject to the convex and concave envelopes of~$\mccormick(\triple)$ associated with each triple~$\triple$ in~$T$. Observe that Condition~\ref{rmpDef1} enforces the linearization of all monomials of two or more variables, and Condition~\ref{rmpDef2} extends the same condition to artificial variables, which always represent the product of two or more original variables.

Given a proper triple set~$T$ for an~\polyopt{}, the associated~\recmcclin{}  is given by
\begin{equation}
\begin{aligned}
   & \min && \sum_{i \in [\nmonomials]} \alpha_i y_{\indexset_i} \\
    & \text{s.t.} && \mccormick(\triple), && \forall \triple \in T  \\
    &&& y_{\indexset} \in [0,1], && \forall \indexset \in \graphnodes.
\end{aligned} \label{polyoptrml}
\end{equation}
Finally, we refer to the size of a~\recmcclin{} as the cardinality of the associated proper triple set~$T$.

\subsection{Full Sequential~\recmcclin{}}\label{sec:sequentiallinearization}

\newcommand{\nonl}{\renewcommand{\nl}{\let\nl\oldnl}}
\newcommand{\pushline}{\Indp}
\newcommand{\popline}{\Indm\dosemic}

Algorithm~\ref{algo:full_sequential} describes the full sequential~\recmcclin{} (\texttt{Seq}), a linearization strategy that is currently used by state-of-the-art global optimization solvers.
\begin{algorithm}
\scriptsize
    

    $T \coloneqq \emptyset$  \quad  \textit{Set of triples}

    \For{$i \in [\nmonomials]$}{
        $A_i \coloneqq \{ \{j\}: j \in \indexset_i  \}$   \quad     \textit{Families of subsets of indices associated with each monomial}
    }
     \While{$\exists A_i: |A_i| > 1$}{

        \textbf{Pick} $\indexset,\indexset' \in A_i$ \textbf{with}  $|A_i| > 1$ \quad         \textit{Select an arbitrary pair of index sets of an arbitrary monomial with $|A_i| > 1$}
        
        $\indexset'' \coloneqq \indexset \cup \indexset' $
        
        $F \coloneqq F \cup \{ \indexset'' \}$
        
        $T \coloneqq T \cup \{(\indexset,\indexset',\indexset'')\}$

        \For{$i' \in [\nmonomials]$ }{
        
            \If{$\{\indexset,\indexset'\} \subseteq A_{i'}$}{
                $A_{i'} \coloneqq A_{i'} \setminus \{\indexset,\indexset'\}$
                
                $A_{i'} \coloneqq A_{i'} \cup \{\indexset''\}$
            }
        }
    }
  \caption{Full Sequential~\recmcclin{} }\label{algo:full_sequential}
\end{algorithm}
\texttt{Seq} is an iterative procedure that, 
in each step, identifies a pair of (original or artificial) variables~$y_{\indexset}$ and~$y_{\indexset'}$ occurring in the same term, where~$\indexset$ and~$\indexset'$ are disjoint subsets of some~$\indexset_i$, and replaces the bilinear product $y_{\indexset} y_{\indexset'}$
for a new auxiliary variable~$y_{\indexset''}$, where~$\indexset'' = \indexset \cup \indexset'$. This substitution is applied to all terms containing both~$y_{\indexset}$ and~$y_{\indexset'}$.
This strategy is termed the \emph{recursive arithmetic interval} in~\cite{RyooSahinidis01}. \texttt{Seq} relies on an (arbitrary) ordering of the variables when deciding on  the bilinear terms that are replaced by auxiliary variables. We show the implications of this behavior in the example below.

\begin{example}  The linearizations  of $f(\bfx) = x_1 x_2 x_3 - x_2 x_3 x_4 - x_1 x_3 x_4$ presented in Figures~\ref{ex:rml1} and~\ref{ex:rml2} can be derived by~\texttt{Seq}. Namely, the linearization in Figure~\ref{ex:rml1} is obtained when \texttt{Seq} adopts the ordering~$(x_1,x_2,x_3,x_4)$, which leads to the substitution of the bilinear terms~$x_1 x_2$, $x_2 x_3$, and $x_1x_3$, in this order. Observe that~$x_2 x_3$ occurs on the first two monomials, but \texttt{Seq} does not do this substitution on both because it replaces~$x_1 x_2$ first. In contrast, the linearization in Figure~\ref{ex:rml2} is derived by \texttt{Seq} based on the ordering~$(x_3,x_4,x_1,x_2)$; first, $x_3 x_4$ is replaced in the last two monomials, and then~$x_1x_3$ is replaced in the first. Therefore, the linearization produced by~\texttt{Seq} is not unique, and as we show in Example~\ref{example2}, both the size and the LP bounds produced by distinct linearizations of~\texttt{Seq} may be different.

\end{example}

\section{Minimum Linearization}\label{sec:minimumlinearization}

Next, we investigate strategies to derive minimum-size~\recmcclin{}s.
In~\S\ref{sec:heuristic} we present a simple greedy approach, which we prove to be suboptimal. In~\S\ref{sec:exact} we  present an exact algorithm to find a minimum-size~\recmcclin{}. Finally, we conclude this section showing that finding a minimum-size~\recmcclin{}  is NP-hard and that a special case  is fixed-parameter tractable.

\subsection{Greedy Linearization}\label{sec:heuristic}

Algorithm~\ref{algo:greedy} describes \texttt{Greedy}, a simple, yet effective, \recmcclin{} strategy that selects in each iteration a pair of (original or artificial) variables~$y_{\indexset}$ and~$y_{\indexset'}$ that appear in as many terms as possible. Then, as in \texttt{Seq}, the bilinear product~$y_{\indexset} y_{\indexset'}$ is replaced in each monomial where it occurs by the artificial variable~$y_{\indexset \cup \indexset'}$. We remark that the main difference between \texttt{Seq} and \texttt{Greedy} is in the selection of pairs; namely, whereas \texttt{Seq} chooses the pairs in an arbitrary way, \texttt{Greedy} tries to reduce as many monomials as possible in each step.
\begin{algorithm}
\scriptsize
\SetAlgoLined
    
    
    $T \coloneqq \emptyset$  \quad  \textit{Set of triples}

    \For{$i \in [\nmonomials]$}{
    
        $A_i \coloneqq \{ \{j\}: j \in \indexset_i  \}$ \quad     \textit{Families of subsets of indices associated with each monomial}
    }
     \While{$\exists A_i: |A_i| > 1$}{
    
        \textbf{Pick $\indexset,\indexset'$ such that $|\{i \in [\nmonomials] : \{ \indexset,\indexset' \} \in A_i\} |$ is maximum} 
        
        $\indexset'' \coloneqq \indexset \cup \indexset' $
        

        $T \coloneqq T \cup \{(\indexset,\indexset',\indexset'')\}$
        
        \For{$i' \in [\nmonomials]$}{
        
            \If{$\{\indexset,\indexset'\} \subseteq A_{i'}$}{
                $A_{i'} \coloneqq A_{i'} \setminus \{\indexset,\indexset'\}$
                
                $A_{i'} \coloneqq A_{i'} \cup \{\indexset''\}$
            }
        }
    }
  \caption{\texttt{Greedy}}\label{algo:greedy}
\end{algorithm}

\texttt{Greedy} frequently performs well,  but worst-case performance can be observed in practice. Example~\ref{ex:vision_greedy} shows why \texttt{Greedy} is outperformed by other linearization strategies on the \texttt{vision} instances, used as benchmark in our experiments (see~\S\ref{sec:computational_results}). More generally, Proposition~\ref{prop:greedy_bad} shows that~\texttt{Greedy} can produce linearizations with  arbitrarily more variables than a minimum-size~\recmcclin{}.

\begin{example}\label{ex:vision_greedy} The \texttt{vision} instances are multilinear polynomials with quadratic, cubic, and quartic terms. The variables represent cells in a  grid, and the quadratic, cubic, and quartic terms are associated with variables forming a diagonal, a right angle, and a square of adjacent cells, respectively. Figure~\ref{ex:visionexample} shows some of the terms in an instance of the problem defined over a 3-by-3 grid. An $n$-by-$n$ instance has $2(n-1)^2$ quadratic terms, $4(n-1)^2$ cubic terms, and $(n-1)^2$ quartic terms.
\newcounter{row}
\newcounter{col}
\newcommand\setrow[3]{
  \setcounter{col}{1}
  \foreach \n in {#1, #2, #3} {
    \edef\x{\value{col} - 0.5}
    \edef\y{2.5 - \value{row}}
    \node[anchor=center] at (\x, \y) {\n};
    \stepcounter{col}
  }
  \stepcounter{row}
}

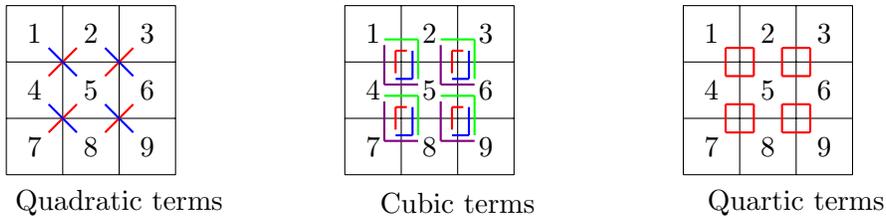
\begin{figure}[ht!]
\centering
\begin{tikzpicture}[scale=0.75]

  \begin{scope}
    \draw (0, 0) grid (3, 3);
    \setrow {1}{2}{3}
    \setrow {4}{5}{6}
    \setrow {7}{8}{9}

    \draw[red,thick] (0.75,0.75) -- (1.25,1.25);
    \draw[blue,thick] (1.25,0.75) -- (0.75,1.25);

    \draw[red,thick] (1.75,0.75) -- (2.25,1.25);
    \draw[blue,thick] (2.25,0.75) -- (1.75,1.25);

    \draw[red,thick] (0.75,1.75) -- (1.25,2.25);
    \draw[blue,thick] (1.25,1.75) -- (0.75,2.25);

    \draw[red,thick] (1.75,1.75) -- (2.25,2.25);
    \draw[blue,thick] (2.25,1.75) -- (1.75,2.25);

    \node[anchor=center] at (2.0, -0.5) {Quadratic terms};
  \end{scope}

    \begin{scope}[xshift=6cm]
        \draw (0, 0) grid (3, 3);
        \setcounter{row}{0}
        \setrow {1}{2}{3}
        \setrow {4}{5}{6}
        \setrow {7}{8}{9}

    \draw[red,thick] (0.9,1.8) -- (0.9,2.2);
    \draw[red,thick] (0.9,2.2) -- (1.1,2.2);    

    \draw[blue,thick] (0.9,1.7) -- (1.2,1.7);
    \draw[blue,thick] (1.2,1.7) -- (1.2,2.2);

    \draw[green,thick] (0.7,2.4) -- (1.3,2.4);
    \draw[green,thick] (1.3,2.4) -- (1.3,1.7);    

    \draw[violet,thick] (0.7,2.3) -- (0.7,1.6);
    \draw[violet,thick] (0.7,1.6) -- (1.3,1.6);

    \draw[red,thick] (0.9,0.8) -- (0.9,1.2);
    \draw[red,thick] (0.9,1.2) -- (1.1,1.2);    

    \draw[blue,thick] (0.9,0.7) -- (1.2,0.7);
    \draw[blue,thick] (1.2,0.7) -- (1.2,1.2);

    \draw[green,thick] (0.7,1.4) -- (1.3,1.4);
    \draw[green,thick] (1.3,1.4) -- (1.3,0.7);    

    \draw[violet,thick] (0.7,1.3) -- (0.7,0.6);
    \draw[violet,thick] (0.7,0.6) -- (1.3,0.6);

    \draw[red,thick] (1.9,1.8) -- (1.9,2.2);
    \draw[red,thick] (1.9,2.2) -- (2.1,2.2);    

    \draw[blue,thick] (1.9,1.7) -- (2.2,1.7);
    \draw[blue,thick] (2.2,1.7) -- (2.2,2.2);

    \draw[green,thick] (1.7,2.4) -- (2.3,2.4);
    \draw[green,thick] (2.3,2.4) -- (2.3,1.7);    

    \draw[violet,thick] (1.7,2.3) -- (1.7,1.6);
    \draw[violet,thick] (1.7,1.6) -- (2.3,1.6);

    \draw[red,thick] (1.9,0.8) -- (1.9,1.2);
    \draw[red,thick] (1.9,1.2) -- (2.1,1.2);    

    \draw[blue,thick] (1.9,0.7) -- (2.2,0.7);
    \draw[blue,thick] (2.2,0.7) -- (2.2,1.2);

    \draw[green,thick] (1.7,1.4) -- (2.3,1.4);
    \draw[green,thick] (2.3,1.4) -- (2.3,0.7);    

    \draw[violet,thick] (1.7,1.3) -- (1.7,0.6);
    \draw[violet,thick] (1.7,0.6) -- (2.3,0.6);





        \node[anchor=center] at (2.0, -0.5) {Cubic terms};
  \end{scope}

    \begin{scope}[xshift=12cm]
        \draw (0, 0) grid (3, 3);
        \setcounter{row}{0}
        \setrow {1}{2}{3}
        \setrow {4}{5}{6}
        \setrow {7}{8}{9}

    \draw[red,thick] (0.75,1.75) -- (0.75,2.25);
    \draw[red,thick] (0.75,2.25) -- (1.25,2.25);    
    \draw[red,thick] (1.25,2.25) -- (1.25,1.75);
    \draw[red,thick] (0.75,1.75) -- (1.25,1.75);

    \draw[red,thick] (0.75,0.75) -- (0.75,1.25);
    \draw[red,thick] (0.75,1.25) -- (1.25,1.25);    
    \draw[red,thick] (1.25,1.25) -- (1.25,0.75);
    \draw[red,thick] (0.75,0.75) -- (1.25,0.75);

    \draw[red,thick] (1.75,1.75) -- (1.75,2.25);
    \draw[red,thick] (1.75,2.25) -- (2.25,2.25);    
    \draw[red,thick] (2.25,2.25) -- (2.25,1.75);
    \draw[red,thick] (1.75,1.75) -- (2.25,1.75);

    \draw[red,thick] (1.75,0.75) -- (1.75,1.25);
    \draw[red,thick] (1.75,1.25) -- (2.25,1.25);    
    \draw[red,thick] (2.25,1.25) -- (2.25,0.75);
    \draw[red,thick] (1.75,0.75) -- (2.25,0.75);

        \node[anchor=center] at (2.0, -0.5) {Quartic terms};
  \end{scope}
\end{tikzpicture}
\caption{All terms in a 3-by-3 example of the \texttt{vision} instances, distinguished by color.} 
\label{ex:visionexample}
\end{figure}
A minimum-size~\recmcclin{} has  one artificial variable per term, starting with the quadratic ones and then proceeding with the cubic and quartic terms.
In contrast, \texttt{Greedy} adds an artificial variable~$y_{\{i,i+1\}}$ for each~$1 \leq i \leq n^2$ such that~$i \mod n \neq 0$ first, representing pairs of cells that appear in the cubic and quartic terms but do not in the quadratic terms. \texttt{Greedy} still needs to add one artificial variable for each term, so the first batch of artificial variables is added in addition to the same number of variables used by a minimum-size~\recmcclin{}.
\end{example}

\begin{proposition}\label{prop:greedy_bad}
The~\recmcclin{} identified by \texttt{Greedy} can be arbitrarily larger than a minimum-size~\recmcclin{}.
\end{proposition}

\subsection{Exact Model}\label{sec:exact}

Next, we present an exact MIP formulation for the identification of minimum-size~\recmcclin{}s, whose feasible solutions form a bijection with the collection of~\recmcclin{}s for an arbitrary~\polyopt{}. 
\begin{subequations}
\begin{align}
    \min
    &\; \sum_{\triple \in \tripleSet} v_\triple \label{minlin.obj} \\
    \text{s.t.} 
    &\; \sum\limits_{\triple \in \tripleSet_i : \head(\triple) = \indexset_i}  u_{i,\triple} = 1 &\,\forall\, i \in [\nmonomials] \textnormal{ with } |\indexset_i| > 1,  \label{minlin.con1} \\
    &\; \sum_{\triple \in \tripleSet_i : \head(\triple) = \graphnode} u_{i,\triple} 
    = \sum_{\triple \in \tripleSet_i : \graphnode \in \tails(\triple) } u_{i,\triple} &\,\forall\, i \in [\nmonomials], \forall\, \graphnode \in \graphnodes_i : 2 \leq |\graphnode| < |\indexset_i|,
    \label{minlin.con2} \\
    &\; u_{i,\triple} \leq v_{\triple} &\,\forall\, \triple \in \tripleSet_i, i \in [\nmonomials], \label{minlin.con3} \\
    &\; \bfu_i \in \B^{|\tripleSet_i|} &\,\forall\, i \in [\nmonomials], \label{minlin.con4} \\
    &\; \bfv \in \B^{|\tripleSet|} & \label{minlin.con5}
\end{align} \label{minlin}
\end{subequations}
The variables of our model are $\bfu_i \in \B^{|\tripleSet_i|}$ and $\bfv \in \B^{|\tripleSet|}$, where:
\begin{itemize}
    \item $u_{i,\triple}$ 
    indicates whether triple $\triple$ is used in the exact linearization of monomial $i$; and
    \item $v_\triple$ 
    indicates whether the triple $\triple$ can be used in the exact linearization of any monomial $i \in [\nmonomials]$.
\end{itemize}
These variables are used to represent the selection of the triples composing a proper triple set~$T$. Variables~$u_{i,\triple}$ are used to select a linearization of monomial~$i$, and variables~$v_\triple$ indicate that a triple  belongs to $T$ and can therefore be used in the linearization of one or more monomials in~$[m]$. 

The constraints~\eqref{minlin.con1}-\eqref{minlin.con2} model conditions~\ref{rmpDef1} and~\ref{rmpDef2}, respectively. Namely, if the index set~$\indexset_{i}$ of monomial~$i$ containing two or more elements, then~$\head(\triple) = \indexset_{i}$ for at least one  triple used in the linearization of~$i$. Similarly, if some index set~$\indexset$ containing two or more elements is such that~$\indexset \in \tails(\triple)$ for some selected triple~$\triple$, then there must exist another selected triple~$\triple'$ such that~$\indexset = \head(\triple')$. The constraints~\eqref{minlin.con3} couple variables~$u_{i,\triple}$ and~$v_\triple$, i.e., if we use~$\triple$ in the linearization of some triple~$\triple$, then we must set~$v_\triple$ to one. Finally, the objective function~\eqref{minlin.obj} minimizes the number of activated triples across the entire~\polyopt{}.

Finally, observe that the variables~$v_\triple$ are necessary in our formulation because a triple~$\triple = (\indexset,\indexset',\indexset'')$ used in the linearization of a monomial~$i$ may not be used in the linearization of another monomial~$i'$ even if both~$\indexset$ and~$\indexset'$ belong to~$\indexset_{i'}$. Observe that this is in agreement with the definition of proper triple sets (see Definition~\ref{def:rmr}), which allows a triple set~$T$ to contain not only a subset that establishes conditions~\ref{rmpDef1} and~\ref{rmpDef2}, but also to include other triples.

\subsection{Complexity and Tractability Results}\label{sec:complexity}

\newcommand{\trinomials}{\mathbb{V}}
\newcommand{\trinomial}{v}
\newcommand{\pairs}{\mathbb{U}}
\newcommand{\pair}{u}

This section investigates the computational complexity of identifying minimum-size~\recmcclin{}s. We restrict our attention to the~3-\polyopt{}s, a special case where all monomials have degree at most 3. We show that finding a minimum-size~\recmcclin{} the 3-\polyopt{} is NP-hard, but also fixed-parameter tractable.


\subsubsection{Dominating Set Formulation of the 3-\polyopt{}}\label{sec:dominatingset}

Let~$f(\bfx)$ be a~3-\polyopt{} with a monomial~$f_i(\bfx) = x_{j}x_{k}  x_{l}$. Any~\recmcclin{} of~$f_i(\bfx)$ must contain one triple $\triple = (x_{j'},x_{k'},y_{\set{j',k'}})$ for some~$\{j',k'\} \subset \{j,k,l\}$, $j' \neq k'$, and one triple~$\triple' = (x_{l'}, y_{\set{j',k'}}, y_{\set{j',k',l'}} )$, $l' \in \{j,k,l\} \setminus \{j',k'\}$, where the first represents the creation of an artificial variable~$y_{\set{j',k'}}$ that linearizes the product of (any) two variables~$x_{j'}$ and~$x_{k'}$ of~$\indexset_i$, and the second linearizes  the product of~$y_{\set{j',k'}}$ and the third  variable~$x_{l'}$. Based on this observation, we can cast an instance~$I'$ of the 3-\polyopt{} as an instance~$I$ of a variation of the dominating set problem over a bipartite graph~$G = (\pairs,\trinomials,E)$. 

Each vertex~$\pair$ of~$\pairs$ is associated with an index set~$\indexset_{\pair}$ that contains with two elements of~$[n]$, and each vertex~$\trinomial$ is associated with an index set~$\indexset_{\trinomial} = \indexset_i$ of some monomial~$i$. We adopt set-theoretical notation to represent the relationships between the elements of~$\pairs$ and~$\trinomials$ based on their associated index sets. For example, $\pair \cap \pair' = \emptyset$ if~$\indexset_\pair \cap \indexset_\trinomial = \emptyset$.
Set~$E$ contains an edge~$(\pair,\trinomial)$ if and only if~$\indexset_\pair \subset \indexset_\trinomial$. For any vertex~$\trinomial$ of~$\trinomials$, we say that the vertices in~$\pairs(\trinomial) \coloneqq \{\pair \in \pairs: (\pair,\trinomial) \in E\}$ \textit{cover}~$\trinomial$. 
The identification of a minimum-size~\recmcclin{} for a~3-\polyopt{} reduces to solving a special case of the dominating set problem on the graph~$G$ constructed as defined above, where \textit{all the dominating vertices must be chosen from~$\pairs$}. See Figure~\ref{fig:dominatingset} for an example of this construction.
\begin{figure}[ht!]
    \centering
        \resizebox{0.45\textwidth}{!}{
\begin{tikzpicture}[main/.style = {draw, circle},square/.style={regular polygon,regular polygon sides=4}]

\node[main] (x1x2)[fill=red] at (0,0) {$x_1  x_2$};
\node[main] (x1x3) at (2,0) {$x_1  x_3$}; 
\node[main] (x2x3) at (4,0) {$x_2  x_3$}; 
\node[main] (x1x4) at (6,0) {$x_1  x_4$};
\node[main] (x2x4) at (8,0) {$x_2  x_4$}; 

\node (x1x2x3) at (3.0,2)[draw] {$x_1 x_2  x_3$};
\node (x1x2x4)[draw] at (6,2) {$x_1  x_2  x_4$};

  \node at (-2.5,2){\textsc{Set $\trinomials$}};
  \node at (-2.5,0){\textsc{Set $\pairs$}};

\path[every node/.style={font=\sffamily\small}]
        (x1x2) edge node [above] {} (x1x2x3)
        (x1x2) edge node [above] {} (x1x2x4)
        (x1x3) edge node [above] {} (x1x2x3)
        (x2x3) edge node [above] {} (x1x2x3)
        (x1x4) edge node [above] {} (x1x2x4)
        (x2x4) edge node [above] {} (x1x2x4)
        ;
\end{tikzpicture} 
}
\caption{Dominating set representation of  $f(\bfx) = x_1 x_2  x_3 + x_1  x_2  x_4$. We have~$\pairs = \{x_1  x_2, x_1  x_3, x_2  x_3, x_1 x_4, x_2  x_4\}$ and
$\trinomials = \{ x_1  x_2  x_3, x_1 x_2  x_4 \}$. Observe that the node associated with~$\indexset_{\set{1,2}}$  covers both nodes in~$\trinomials$. }
    \label{fig:dominatingset}
\end{figure}
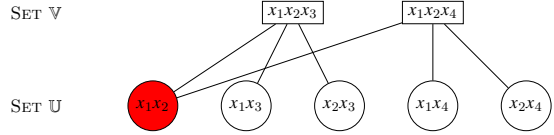


\subsubsection{Reduction Rules and NP-hardness}\label{sec:reductionrules}

The connection with the dominating set problem allows us to derive a set of reduction rules to remove elements from~$\pairs$ and~$\trinomials$.
The sequential and iterative application of these rules 
yield a kernelization algorithm, as used
parameterized complexity theory (see e.g., \cite{fomin2019kernelization}). 

\begin{theorem}[Reduction Rules]\label{thm:reductionrules} The application of the following set of rules preserves at least one dominating set in~$G$ associated with a minimum linearization of~$f(\bfx)$:

\begin{enumerate}[align=left,label=\textbf{Rule \arabic*}]
\item \label{rule1} For every~$\trinomial \in \trinomials$ such that~$\trinomial \cap \trinomial' = \emptyset$ for every~$\trinomial' \in \trinomials \setminus \{\trinomial\}$:
\begin{itemize}
    \item Select an arbitrary pair~$\pair \in \pairs(\trinomial)$;
    \item Remove~$\trinomial$ from~$\trinomials$ and all elements of~$\pairs(\trinomial)$ from~$\pairs$. 
\end{itemize}
\item \label{rule2}
Remove all elements of~$\pairs$ of degree 1.
\item \label{rule3} For each element~$\trinomial$ of~$\trinomials$ with a single neighbor~$\pair$, select~$\pair$.
\item \label{rule4} Remove all elements of~$\pairs$ without neighbors.
\item \label{rule5} The problem can be decomposed by its connected components in~$G$.
\end{enumerate}
\end{theorem}

\begin{example}
We illustrate the application of the reduction rules on~$f(\bfx) = x_1   x_2   x_3 + x_4   x_5   x_6 + x_4   x_6   x_8 + x_7   x_8   x_9 + x_8   x_9   x_{10} + x_7   x_9   x_{10}$ in Figure~\ref{fig:reductionexample}. 
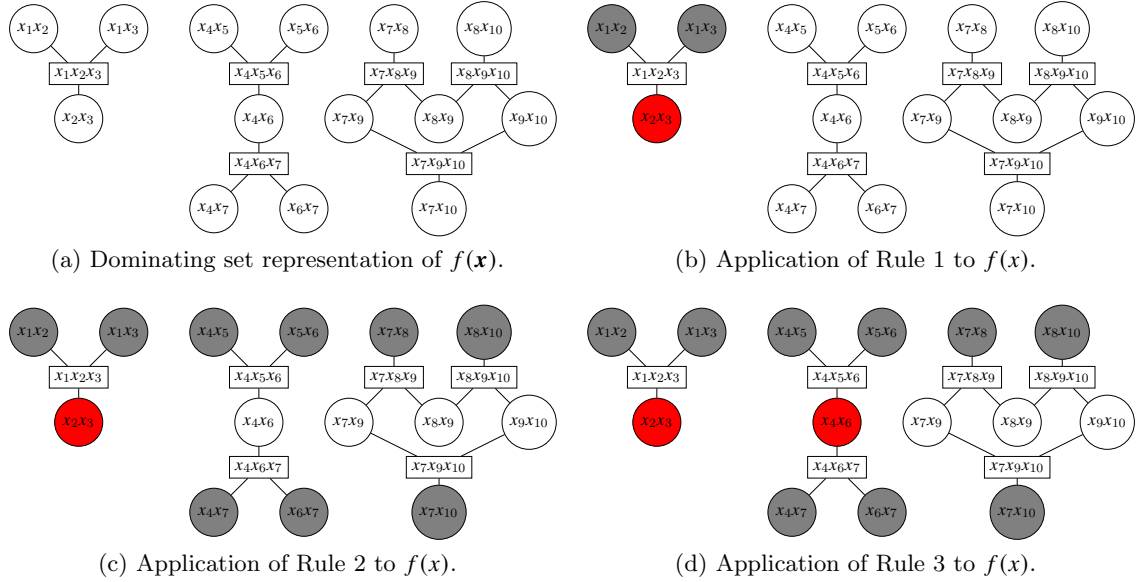
\begin{figure}[ht!]
\centering

\subfloat[Dominating set representation of $f(\bfx)$.
    \label{fig:ex1}]{
\resizebox{0.45\textwidth}{!}{
\begin{tikzpicture}[main/.style = {draw, circle},square/.style={regular polygon,regular polygon sides=4}] 
\node (x1x2x3)[,draw] at (0,1) {$x_1   x_2   x_3$};
\node[main] (x1x2) at (-1,2) {$x_1   x_2$}; 
\node[main] (x2x3) at (0,0) {$x_2   x_3$}; 
\node[main] (x1x3) at (1,2) {$x_1   x_3$};

\node (x4x5x6)[,draw] at (4,1) {$x_4   x_5   x_6$};
\node[main] (x4x5) at (3,2) {$x_4   x_5$}; 
\node[main] (x4x6) at (4,0) {$x_4   x_6$}; 
\node[main] (x5x6) at (5,2) {$x_5   x_6$}; 

\node (x4x6x7)[,draw] at (4,-1) {$x_4   x_6   x_7$};

\node[main] (x4x7) at (3,-2) {$x_4   x_7$}; 
\node[main] (x6x7) at (5,-2) {$x_6   x_7$};

\node[main] (x7x8) at (7,2) {$x_7   x_8$}; 
\node[main] (x8x10) at (9,2) {$x_8   x_{10}$}; 
\node[main] (x7x10) at (8,-2) {$x_7   x_{10}$};

\node[main] (x7x9) at (6,0) {$x_7   x_9$}; 
\node[main] (x8x9) at (8,0) {$x_8   x_{9}$}; 
\node[main] (x9x10) at (10,0) {$x_9   x_{10}$};

\node (x7x8x9)[,draw] at (7,1) {$x_7   x_8   x_9$};
\node (x8x9x10)[,draw] at (9,1) {$x_8   x_9   x_{10}$};
\node (x7x9x10)[,draw] at (8,-1) {$x_7   x_9   x_{10}$};

\path[every node/.style={font=\sffamily\small}]
        (x1x2) edge node [above] {} (x1x2x3)
        (x2x3) edge node [above] {} (x1x2x3)
        (x1x3) edge node [above] {} (x1x2x3)
        (x4x5) edge node [above] {} (x4x5x6)
        (x5x6) edge node [above] {} (x4x5x6)
        (x4x6) edge node [above] {} (x4x5x6)
        (x4x6) edge node [above] {} (x4x6x7)
        (x4x7) edge node [above] {} (x4x6x7)
        (x6x7) edge node [above] {} (x4x6x7)
        (x7x8) edge node [above] {} (x7x8x9)
        (x8x10) edge node [above] {} (x8x9x10)
        (x7x9) edge node [above] {} (x7x8x9)
        (x8x9) edge node [above] {} (x7x8x9)
        (x8x9) edge node [above] {} (x8x9x10)
        (x9x10) edge node [above] {} (x8x9x10)
        (x7x9) edge node [above] {} (x7x9x10)
        (x9x10) edge node [above] {} (x7x9x10)
        (x7x10) edge node [above] {} (x7x9x10)
        ;

\end{tikzpicture} 
}
}
\subfloat[Application of Rule 1 to $f(x)$.
    \label{fig:ex2}]{
        \resizebox{0.45\textwidth}{!}        {
\begin{tikzpicture}[main/.style = {draw, circle},square/.style={regular polygon,regular polygon sides=4}]

\node (x1x2x3)[,draw] at (0,1) {$x_1   x_2   x_3$};
\node[main] (x1x2)[fill=gray] at (-1,2) {$x_1   x_2$}; 
\node[main] (x2x3)[fill=red] at (0,0) {$x_2   x_3$}; 
\node[main] (x1x3)[fill=gray] at (1,2) {$x_1   x_3$};

\node (x4x5x6)[,draw] at (4,1) {$x_4   x_5   x_6$};
\node[main] (x4x5) at (3,2) {$x_4   x_5$}; 
\node[main] (x4x6) at (4,0) {$x_4   x_6$}; 
\node[main] (x5x6) at (5,2) {$x_5   x_6$}; 

\node (x4x6x7)[,draw] at (4,-1) {$x_4   x_6   x_7$};

\node[main] (x4x7) at (3,-2) {$x_4   x_7$}; 
\node[main] (x6x7) at (5,-2) {$x_6   x_7$};

\node[main] (x7x8) at (7,2) {$x_7   x_8$}; 
\node[main] (x8x10) at (9,2) {$x_8   x_{10}$}; 
\node[main] (x7x10) at (8,-2) {$x_7   x_{10}$};

\node[main] (x7x9) at (6,0) {$x_7   x_9$}; 
\node[main] (x8x9) at (8,0) {$x_8   x_{9}$}; 
\node[main] (x9x10) at (10,0) {$x_9   x_{10}$};

\node (x7x8x9)[,draw] at (7,1) {$x_7   x_8   x_9$};
\node (x8x9x10)[,draw] at (9,1) {$x_8   x_9   x_{10}$};
\node (x7x9x10)[,draw] at (8,-1) {$x_7   x_9   x_{10}$};

\path[every node/.style={font=\sffamily\small}]
        (x1x2) edge node [above] {} (x1x2x3)
        (x2x3) edge node [above] {} (x1x2x3)
        (x1x3) edge node [above] {} (x1x2x3)
        (x4x5) edge node [above] {} (x4x5x6)
        (x5x6) edge node [above] {} (x4x5x6)
        (x4x6) edge node [above] {} (x4x5x6)
        (x4x6) edge node [above] {} (x4x6x7)
        (x4x7) edge node [above] {} (x4x6x7)
        (x6x7) edge node [above] {} (x4x6x7)
        (x7x8) edge node [above] {} (x7x8x9)
        (x8x10) edge node [above] {} (x8x9x10)
        (x7x9) edge node [above] {} (x7x8x9)
        (x8x9) edge node [above] {} (x7x8x9)
        (x8x9) edge node [above] {} (x8x9x10)
        (x9x10) edge node [above] {} (x8x9x10)
        (x7x9) edge node [above] {} (x7x9x10)
        (x9x10) edge node [above] {} (x7x9x10)
        (x7x10) edge node [above] {} (x7x9x10)
        ;

\end{tikzpicture} 
}
}

\subfloat[Application of Rule 2 to $f(x)$.
    \label{fig:ex3}]{
        \resizebox{0.45\textwidth}{!}{
\begin{tikzpicture}[main/.style = {draw, circle},square/.style={regular polygon,regular polygon sides=4}]

\node (x1x2x3)[,draw] at (0,1) {$x_1   x_2   x_3$};
\node[main] (x1x2)[fill=gray] at (-1,2) {$x_1   x_2$}; 
\node[main] (x2x3)[fill=red] at (0,0) {$x_2   x_3$}; 
\node[main] (x1x3)[fill=gray] at (1,2) {$x_1   x_3$};

\node (x4x5x6)[,draw] at (4,1) {$x_4   x_5   x_6$};
\node[main] (x4x5)[fill=gray] at (3,2) {$x_4   x_5$}; 
\node[main] (x4x6) at (4,0) {$x_4   x_6$}; 
\node[main] (x5x6)[fill=gray] at (5,2) {$x_5   x_6$}; 

\node (x4x6x7)[,draw] at (4,-1) {$x_4   x_6   x_7$};

\node[main] (x4x7)[fill=gray] at (3,-2) {$x_4   x_7$}; 
\node[main] (x6x7)[fill=gray] at (5,-2) {$x_6   x_7$};

\node[main] (x7x8)[fill=gray] at (7,2) {$x_7   x_8$}; 
\node[main] (x8x10)[fill=gray] at (9,2) {$x_8   x_{10}$}; 
\node[main] (x7x10)[fill=gray] at (8,-2) {$x_7   x_{10}$};

\node[main] (x7x9) at (6,0) {$x_7   x_9$}; 
\node[main] (x8x9) at (8,0) {$x_8   x_{9}$}; 
\node[main] (x9x10) at (10,0) {$x_9   x_{10}$};

\node (x7x8x9)[,draw] at (7,1) {$x_7   x_8   x_9$};
\node (x8x9x10)[,draw] at (9,1) {$x_8   x_9   x_{10}$};
\node (x7x9x10)[,draw] at (8,-1) {$x_7   x_9   x_{10}$};

\path[every node/.style={font=\sffamily\small}]
        (x1x2) edge node [above] {} (x1x2x3)
        (x2x3) edge node [above] {} (x1x2x3)
        (x1x3) edge node [above] {} (x1x2x3)
        (x4x5) edge node [above] {} (x4x5x6)
        (x5x6) edge node [above] {} (x4x5x6)
        (x4x6) edge node [above] {} (x4x5x6)
        (x4x6) edge node [above] {} (x4x6x7)
        (x4x7) edge node [above] {} (x4x6x7)
        (x6x7) edge node [above] {} (x4x6x7)
        (x7x8) edge node [above] {} (x7x8x9)
        (x8x10) edge node [above] {} (x8x9x10)
        (x7x9) edge node [above] {} (x7x8x9)
        (x8x9) edge node [above] {} (x7x8x9)
        (x8x9) edge node [above] {} (x8x9x10)
        (x9x10) edge node [above] {} (x8x9x10)
        (x7x9) edge node [above] {} (x7x9x10)
        (x9x10) edge node [above] {} (x7x9x10)
        (x7x10) edge node [above] {} (x7x9x10)
        ;

\end{tikzpicture}
}
}
\subfloat[Application of Rule 3 to $f(x)$.
    \label{fig:ex4}]{
    \resizebox{0.45\textwidth}{!}{
\begin{tikzpicture}[main/.style = {draw, circle},square/.style={regular polygon,regular polygon sides=4}]

\node (x1x2x3)[,draw] at (0,1) {$x_1   x_2   x_3$};
\node[main] (x1x2)[fill=gray] at (-1,2) {$x_1   x_2$}; 
\node[main] (x2x3)[fill=red] at (0,0) {$x_2   x_3$}; 
\node[main] (x1x3)[fill=gray] at (1,2) {$x_1   x_3$};

\node (x4x5x6)[,draw] at (4,1) {$x_4   x_5   x_6$};
\node[main] (x4x5)[fill=gray] at (3,2) {$x_4   x_5$}; 
\node[main] (x4x6)[fill=red] at (4,0) {$x_4   x_6$}; 
\node[main] (x5x6)[fill=gray] at (5,2) {$x_5   x_6$}; 

\node (x4x6x7)[,draw] at (4,-1) {$x_4   x_6   x_7$};

\node[main] (x4x7)[fill=gray] at (3,-2) {$x_4   x_7$}; 
\node[main] (x6x7)[fill=gray] at (5,-2) {$x_6   x_7$};

\node[main] (x7x8)[fill=gray] at (7,2) {$x_7   x_8$}; 
\node[main] (x8x10)[fill=gray] at (9,2) {$x_8   x_{10}$}; 
\node[main] (x7x10)[fill=gray] at (8,-2) {$x_7   x_{10}$};

\node[main] (x7x9) at (6,0) {$x_7   x_9$}; 
\node[main] (x8x9) at (8,0) {$x_8   x_{9}$}; 
\node[main] (x9x10) at (10,0) {$x_9   x_{10}$};

\node (x7x8x9)[,draw] at (7,1) {$x_7   x_8   x_9$};
\node (x8x9x10)[,draw] at (9,1) {$x_8   x_9   x_{10}$};
\node (x7x9x10)[,draw] at (8,-1) {$x_7   x_9   x_{10}$};

\path[every node/.style={font=\sffamily\small}]
        (x1x2) edge node [above] {} (x1x2x3)
        (x2x3) edge node [above] {} (x1x2x3)
        (x1x3) edge node [above] {} (x1x2x3)
        (x4x5) edge node [above] {} (x4x5x6)
        (x5x6) edge node [above] {} (x4x5x6)
        (x4x6) edge node [above] {} (x4x5x6)
        (x4x6) edge node [above] {} (x4x6x7)
        (x4x7) edge node [above] {} (x4x6x7)
        (x6x7) edge node [above] {} (x4x6x7)
        (x7x8) edge node [above] {} (x7x8x9)
        (x8x10) edge node [above] {} (x8x9x10)
        (x7x9) edge node [above] {} (x7x8x9)
        (x8x9) edge node [above] {} (x7x8x9)
        (x8x9) edge node [above] {} (x8x9x10)
        (x9x10) edge node [above] {} (x8x9x10)
        (x7x9) edge node [above] {} (x7x9x10)
        (x9x10) edge node [above] {} (x7x9x10)
        (x7x10) edge node [above] {} (x7x9x10)
        ;

\end{tikzpicture} 
}
}

\caption{Reduction rules applied to $f(x) \coloneqq x_1   x_2   x_3 + x_4   x_5   x_6 + x_4   x_6   x_8 + x_7   x_8   x_9 + x_8   x_9   x_{10} + x_7   x_9   x_{10}$.}
\label{fig:reductionexample}
\end{figure}
The dominating set formulation of~$f(\bfx)$ is depicted in Figure~\ref{fig:ex1}. Nodes of~$\pairs$ incorporated to the optimal solution are shaded in red; eliminated nodes are shaded in gray. The monomial  $f_1(\bfx) = x_1   x_2   x_3$ does not share variables with other monomials, so  we can apply~\ref{rule1} and  select~$x_2   x_3$ to cover $x_1   x_2   x_3$ while excluding~$x_1   x_2$ and~$x_1   x_3$ (see Figure~\ref{fig:ex2}). Next, \ref{rule2} eliminates~$x_4   x_5$, $x_5   x_6$, $x_4   x_7$, $x_6   x_7$, $x_7   x_8$, $x_8   x_{10}$, and $x_7   x_{10}$ (see Figure~\ref{fig:ex3}). Finally, Figure~\ref{fig:ex4} shows the result of~\ref{rule3}, where we select~$x_4   x_6$ to cover both~$x_4   x_5   x_6$ and~$x_4   x_6   x_7$.
\end{example}

\begin{proposition}[Structure of the Kernel]\label{prop:structurereduced}
Let~$G^r = (\pairs^r,\trinomials^r,E^r)$ denote the graph resulting from the exhaustive application of the rules in Theorem~\ref{thm:reductionrules}. $G^r$ satisfies the following properties:
\begin{enumerate}[align=left,label=\textbf{Property \arabic*}]
\item \label{property1}
Each element of~$\trinomials^r$ has 2 or 3 neighbors in~$\pairs^r$.
\item \label{property2}
Each element of~$\pairs^r$ has at least 2 neighbors in~$\trinomials^r$.
\item \label{property3}
$G^r$ is a~$K_{2,2}-$free graphs.
\item \label{property4}
If~$\pair$ is not selected, then any solution must contain one~$\pair'$ for each neighbor of~$\pair$.
\end{enumerate}
\end{proposition}

The structural results presented in Proposition~\ref{prop:structurereduced} allow us to derive a connection  of the 3-\polyopt{} with the vertex cover problem. We explore this connection to show that \polyopt{} is NP-hard. 
\begin{theorem}\label{thm:nphard}
The minimum linearization of 3-\polyopt{} is NP-hard.
\end{theorem}

We conclude by showing that the 3-\polyopt{} is fixed-parameter tractable. Namely, given a fixed value~$k$ and the graph~$G^r = (\pairs^r,\trinomials^r,E)$ associated with a reduced instance~$I$ of the 3-\polyopt{}, one may decide whether~$I$ admits a linearization with at most~$k$ elements in time~$O(k^6 + 3^kk^2)$.
\begin{theorem}\label{thm:fpt}
The 3-\polyopt{} is fixed-parameter tractable.
\end{theorem}

\section{Best Bound LP Relaxation}\label{sec:BB}

Let~$\bfvhat \in \B^{|\tripleSet|}$ be a binary indicator vector representing a proper triple set~$T$, i.e., $\hat{v}_\triple = 1$ if and only if~$\triple \in T$. For a given~$\bfvhat \in \B^{|\tripleSet|}$, the formulation  presented in~\eqref{polyoptrml} can be rewritten as the following linear program (LP):
\begin{subequations}
\begin{align}
    \min\limits_{\bfy \in [0,1]^{|\graphnodes|}} 
    &\; \sum_{\graphnode \in \graphnodes} \beta_\graphnode y_{\graphnode}
    \label{lp.obj} \\
    \text{s.t.} &\; 
    \underbrace{\begin{pmatrix}
        -1 & 0 & 1 \\
        0 & -1 & 1 \\
        1 & 1 & -1 
    \end{pmatrix}}_{\eqqcolon B} 
    \underbrace{\begin{pmatrix}
        y_\graphnode \\
        y_{\graphnode'} \\
        y_{\graphnode''}
    \end{pmatrix}}_{\eqqcolon \bfy_\triple}
    \leq 
    \underbrace{\begin{pmatrix} 
        1 \\
        1 \\
        2
    \end{pmatrix}}_{\coloneqq \textbf{b}} + 
    \underbrace{\begin{pmatrix}
        -1 \\
        -1 \\
        -1
    \end{pmatrix}}_{\eqqcolon \textbf{c}} \hat{v}_\triple &\,\forall\, \triple \in \tripleSet \label{lp.con1} \\
    & y_{\graphnode}  \leq 1 &\,\forall\, \graphnode \in \graphnodes \label{lp.con2}
\end{align} \label{lp}
\end{subequations}
Coefficient~$\beta_\graphnode = \alpha_i$ if $\indexset = \indexset_i$ for some~$i \in [m]$ 
and $\beta_\graphnode = 0$ otherwise. Lemma~\ref{lemma:LPisBnLD} shows that the optimal solution to~\eqref{lp.con1} is bounded for any choice of $\bfvhat$.


\begin{lemma}\label{lemma:LPisBnLD}
The optimal objective value of~\eqref{lp} lies in the interval $[-\eta,0]$ where $\eta = -\sum\limits_{i \in [\nmonomials]} \min(0,\alpha_i)$. 
\end{lemma}
\begin{proof}
Observe that the optimal objective value of~\eqref{lp}
must be less than or equal to zero since $\bfy = 0$ is feasible for any choice of~$\bfvhat$. A lower bound of $\sum\limits_{\graphnode \in \graphnodes} \min(0,\beta_\graphnode)$ 
can be attained by setting $y_{\indexset} = 1$ if $\beta_\indexset < 0$ and $y_{\indexset} = 0$ if $\beta_\indexset \geq 0$.  Since $\beta_\indexset \neq 0$ only for $\indexset \in \{\indexset_1,\ldots,\indexset_\nmonomials\}$, the lower bound simplifies to $\sum\limits_{i \in [\nmonomials]} \min(0,\alpha_i)$, so the result follows. 
\end{proof}

Let $\bflambda_\triple \coloneqq (\lambda_{\triple,1},\lambda_{\triple,2},\lambda_{\triple,3})$ denote the vector with 
the dual multipliers for the three inequalities in~\eqref{lp.con1} associated with the triple $\triple$.  Let $\mu_\graphnode$ denote the multiplier for the bound constraint $y_\graphnode \leq 1$ in~\eqref{lp.con2} associated with the index set $\graphnode$. Moreover,  
let~$\bflambda$ be a vector containing~$\bflambda_\triple$ for all    $\triple \in \tripleSet$, 
and $\bfmu$ denote the collection of
multipliers $\mu_\graphnode$ for all $\graphnode \in \graphnodes$. 
The dual of~\eqref{lp} can be written as follows:
\begin{subequations}
\begin{align}
    \max\limits_{(\bflambda,\bfmu) \in \R^{|\tripleSet|} \times \R^{|\graphnodes|}} 
    &\;\; Obj(\bflambda,\bfmu) \coloneqq  -\sum\limits_{\triple \in \tripleSet} (\textbf{b}^T\bflambda_\triple +\textbf{c}^T\bflambda_\triple \vhat_{\triple}) - \sum\limits_{\graphnode \in \graphnodes} \mu_\graphnode &\, \label{dual.obj} \\
    \text{s.t.} 
    &\; \beta_\graphnode + \sum\limits_{\triple : \graphnode = \tailone(\triple)}
    (-\lambda_{\triple,1} + \lambda_{\triple,3}) +
    \sum\limits_{\triple : \graphnode = \tailtwo(\triple)}
    (-\lambda_{\triple,2} + \lambda_{\triple,3}) \nonumber \\
    &\;\; + \sum\limits_{\triple : \graphnode = \head(\triple)} 
    (\lambda_{\triple,1} + \lambda_{\triple,2} - \lambda_{\triple,3})
    + \mu_\graphnode \geq 0 &\,\forall\, \graphnode \in \graphnodes 
    \label{dual.con1} \\
    &\;\; \bflambda,\bfmu \geq 0 &\, \label{dual.con3}
\end{align} \label{dual}
\end{subequations}
 Similarly to the LP~\eqref{lp}, the optimal solution to~\eqref{dual} is bounded given a fixed~$\bfvhat$.
\begin{lemma}\label{lemma:bnLDualobj}
The optimal objective value of~\eqref{dual} lies in the interval $[-\eta,0]$.
\end{lemma}
\begin{proof}
This follows from Lemma~\ref{lemma:LPisBnLD} and strong duality. 
\end{proof}


%
\subsection{Bounds on the dual multipliers}

In this section, we prove that, for any given~$\bfvhat$, $\optbflambda{\bfvhat}$ and~$\optbfmu{\bfvhat}$ have finite bounds that are independent of $\bfvhat$; this result is stated in Proposition~\ref{prop:dual_bounds} and follows from Lemmas~\ref{lemma:complSlackness}, $\ref{lemma:bndlambda3}$, and \ref{lemma:bndlambda12}. For the purposes of this section we assume that~$\bfvhat$ is fixed and omit it from the notation for brevity.


\begin{proposition}\label{prop:dual_bounds}
Let $\optbflambda{\bfvhat},\optbfmu{\bfvhat}$ denote an optimal solution to~\eqref{dual} for a fixed~$\bfvhat$. We have
$\lambda^{\bfvhat}_{\triple,j} \leq M_{\triple,j}$ for all $\triple \in \tripleSet$, $j = 1,2,3$ with $M_{\triple,j} \in [0,\infty)$ and $\mu^{\bfvhat}_\graphnode \leq M_\graphnode$ for all $\graphnode \in \graphnodes$ with $M_\graphnode \in [0,\infty)$.
\end{proposition}



\begin{lemma}
\label{lemma:complSlackness}
There exists an optimal solution $\optbflambda{\bfvhat},\optbfmu{\bfvhat}$ to~\eqref{dual} with $\optbflambda{\bfvhat}_\triple = 0$ for all $\triple$ such that $\vhat_\triple = 0$.
\end{lemma}
\begin{proof}
Suppose $(\bflambdatilde,\bfmutilde)$ is  feasible for~\eqref{dual} and $\bflambdatilde_{\triple'} \neq 0$ for some $\triple'$ such that $\vhat_{\triple'} = 0$.  The result follows from the fact that there exists a solution~$(\bflambdahat,\bfmuhat)$ that is feasible for~\eqref{dual}  with $\bflambdahat_{\triple'} = 0$ and $Obj(\bflambdahat,\bfmuhat) = Obj(\bflambdatilde,\bfmutilde)$. Let $\triple' = (\graphnode',\graphnode'',\graphnode''')$.  
Define $\bflambdahat$ and $\bfmuhat$ as follows:
\begin{equation}
\begin{aligned}
    \bflambdahat_{\triple} = \left\{ \begin{aligned}
        \bflambdatilde_{\triple} & \text{ if } \triple \neq \triple' \\
        0 & \text{ if } \triple = \triple'
    \end{aligned} \right. \text{  and  }
    \muhat_\graphnode = \left\{\begin{aligned}
        \mutilde_\graphnode & \text{ if } \graphnode \notin \{\graphnode',\graphnode'',\graphnode'''\} \\
        \mutilde_{\graphnode'} + \lambdatilde_{\triple',3} 
        & \text{ if } \graphnode = \graphnode' \\
        \mutilde_{\graphnode''} + \lambdatilde_{\triple',3} 
        & \text{ if } \graphnode = \graphnode'' \\
        \mutilde_{\graphnode'''} + \lambdatilde_{\triple',1} 
        + \lambdatilde_{\triple',2} & \text{ if } \graphnode = \graphnode'''.
    \end{aligned} \right.
\end{aligned}\label{defmultiplierhat}
\end{equation}
By construction, we have
$\bflambdahat, \bfmuhat \geq 0$. Moreover, $\bflambdahat_{\triple'}$, $\bfmuhat_{\graphnode'}$, $\bfmuhat_{\graphnode''}$, and~$\bfmuhat_{\graphnode'''}$ only figure in the inequalities in~\eqref{dual.con1} for $\graphnode \in \{\graphnode',\graphnode'',\graphnode'''\}$.  Hence, the inequality~\eqref{dual.con1} holds for all $\graphnode \setminus \{\graphnode',\graphnode'',\graphnode'''\}$.  Therefore, we just need to show that~$(\bflambdahat, \bfmuhat)$  satisfy~\eqref{dual.con1} for $\{\graphnode',\graphnode'',\graphnode'''\}$.

First, consider the left hand side of~\eqref{dual.con1} for $\graphnode = \graphnode'$; the analysis for~$\graphnode = \graphnode''$ is identical.  
We have
    \begin{subequations}
    \begin{align}
        &\; \beta_{\graphnode'} + \sum\limits_{\triple : \graphnode' = \tailone(\triple)}
    (-\lambdahat_{\triple,1} + \lambdahat_{\triple,3}) +
    \sum\limits_{\triple : \graphnode' = \tailtwo(\triple)}
    (-\lambdahat_{\triple,2} + \lambdahat_{\triple,3}) 
    + \sum\limits_{\triple : \graphnode' = \head(\triple)} 
    (\lambdahat_{\triple,1} + \lambdahat_{\triple,2} - \lambdahat_{\triple,3})
    + \muhat_{\graphnode'} \label{ineq1.1} \\
    =&\; \beta_{\graphnode'} + (-\lambdahat_{\triple',1} + \lambdahat_{\triple',3}) + \sum\limits_{\triple : \graphnode' = \tailone(\triple) \setminus \{\triple'\}}
    (-\lambdahat_{\triple,1} + \lambdahat_{\triple,3}) +
    \sum\limits_{\triple : \graphnode' = \tailtwo(\triple)}
    (-\lambdahat_{\triple,2} + \lambdahat_{\triple,3}) \nonumber \\
    &\; + \sum\limits_{\triple : \graphnode' = \head(\triple)} 
    (\lambdahat_{\triple,1} + \lambdahat_{\triple,2} - \lambdahat_{\triple,3})
    + \muhat_{\graphnode'} \label{ineq1.2} \\
    =&\; \beta_{\graphnode'} + 0 + \sum\limits_{\triple : \graphnode' = \tailone(\triple) \setminus \{\triple'\}} 
    (-\lambdatilde_{\triple,1} + \lambdatilde_{\triple,3}) +
    \sum\limits_{\triple : \graphnode' = \tailtwo(\triple)}
    (-\lambdatilde_{\triple,2} + \lambdatilde_{\triple,3}) \nonumber \\
    &\; + \sum\limits_{\triple : \graphnode' = \head(\triple)} 
    (\lambdatilde_{\triple,1} + \lambdatilde_{\triple,2} - \lambdatilde_{\triple,3})
    + \mutilde_{\graphnode'} + \lambdatilde_{\triple',3} \label{ineq1.3} \\
    =&\; \beta_{\graphnode'} + \sum\limits_{\triple : \graphnode' = \tailone(\triple) }
    (-\lambdatilde_{\triple,1} + \lambdatilde_{\triple,3}) +
    \sum\limits_{\triple : \graphnode' = \tailtwo(\triple)}
    (-\lambdatilde_{\triple,2} + \lambdatilde_{\triple,3}) \nonumber \\
    &\; + \sum\limits_{\triple : \graphnode' = \head(\triple)} 
    (\lambdatilde_{\triple,1} + \lambdatilde_{\triple,2} - \lambdatilde_{\triple,3})
    + \mutilde_{\graphnode'} + \lambdatilde_{\triple',1} \label{ineq1.4} \\
    \geq &\; \lambdatilde_{\triple',1} \geq 0 \label{ineq1.5}
    \end{align}
    \end{subequations}
     In the first equality, we move the terms associated with~$\triple'$ out of the first summation.  The equality in~\eqref{ineq1.3} is obtained by substituting~\eqref{defmultiplierhat}, and~\eqref{ineq1.4} follows by adding and subtracting the term $\lambdatilde_{\triple',1}$ and collecting the term $(-\lambdatilde_{\triple',1} + \lambdatilde_{\triple',3})$ into the first summation.  
    The first inequality in~\eqref{ineq1.5} is obtained from~\eqref{dual.con1} holding for the variables $(\bflambdatilde,\bfmutilde)$, and the final inequality follows from $\bflambdatilde \geq 0$.  
    
    For $\graphnode = \graphnode'''$, we have
    \begin{subequations}
    \begin{align}
        &\; \beta_{\graphnode'''} + \sum\limits_{\triple : \graphnode''' = \tailone(\triple)}
    (-\lambdahat_{\triple,1} + \lambdahat_{\triple,3}) +
    \sum\limits_{\triple : \graphnode''' = \tailtwo(\triple)}
    (-\lambdahat_{\triple,2} + \lambdahat_{\triple,3}) 
    + \sum\limits_{\triple : \graphnode''' = \head(\triple)} 
    (\lambdahat_{\triple,1} + \lambdahat_{\triple,2} - \lambdahat_{\triple,3})
    + \muhat_{\graphnode'''} \label{ineq3.1} \\
    =&\; \beta_{\graphnode'''} + \sum\limits_{\triple : \graphnode''' = \tailone(\triple)}
    (-\lambdahat_{\triple,1} + \lambdahat_{\triple,3}) +
    \sum\limits_{\triple : \graphnode''' = \tailtwo(\triple)}
    (-\lambdahat_{\triple,2} + \lambdahat_{\triple,3}) \nonumber \\
    &\; +     (\lambdahat_{\triple',1} + \lambdahat_{\triple',2} - \lambdahat_{\triple',3}) + \sum\limits_{\triple : \graphnode''' = \head(\triple) \setminus \{\triple'\}} 
    (\lambdahat_{\triple,1} + \lambdahat_{\triple,2} - \lambdahat_{\triple,3}) + 
    \muhat_{\graphnode'''} \label{ineq3.2} \\
    =&\; \beta_{\graphnode'''} + \sum\limits_{\triple : \graphnode' = \tailone(\triple)}
    (-\lambdatilde_{\triple,1} + \lambdatilde_{\triple,3}) +
    \sum\limits_{\triple : \graphnode''' = \tailtwo(\triple)  \setminus \{\triple'\}}
    (-\lambdatilde_{\triple,2} + \lambdatilde_{\triple,3}) \nonumber \\
    &\; + \sum\limits_{\triple : \graphnode'' = \head(\triple) \setminus \{\triple'\}} 
    (\lambdatilde_{\triple,1} + \lambdatilde_{\triple,2} - \lambdatilde_{\triple,3}) + 0 
    + \mutilde_{\graphnode'''} + \lambdatilde_{\triple',1} + \lambdatilde_{\triple',2} \label{ineq3.3} \\
    =&\; \beta_{\graphnode'''} + \sum\limits_{\triple : \graphnode''' = \tailone(\triple) }
    (-\lambdatilde_{\triple,1} + \lambdatilde_{\triple,3}) +
    \sum\limits_{\triple : \graphnode''' = \tailtwo(\triple)}
    (-\lambdatilde_{\triple,2} + \lambdatilde_{\triple,3}) \nonumber \\
    &\; + \sum\limits_{\triple : \graphnode''' = \head(\triple)} 
    (\lambdatilde_{\triple,1} + \lambdatilde_{\triple,2} - \lambdatilde_{\triple,3})
    + \mutilde_{\graphnode'''} + \lambdatilde_{\triple',3} \label{ineq3.4} \\
    \geq &\; \lambdatilde_{\triple',3} \geq 0 \label{ineq3.5}
    \end{align}
    \end{subequations}
    Equality~\eqref{ineq3.2} follows from splitting the last summation over~$\triple'$. The equality in~\eqref{ineq3.3} is obtained by substituting~\eqref{defmultiplierhat}, and~\eqref{ineq3.4} follows by adding and subtracting the term $\lambdatilde_{\triple',3}$ and collecting the term $(\lambdatilde_{\triple',1} + \lambdatilde_{\triple',2} - \lambdatilde_{\triple',3})$ into the last summation.    The first inequality in~\eqref{ineq3.5} is obtained from~\eqref{dual.con1} holding for the variables $(\bflambdatilde,\bfmutilde)$, and the final inequality follows from $\bflambdatilde \geq 0$. Therefore, $(\bflambdahat,\bfmuhat)$ is  feasible for~\eqref{dual.con1} and $\bflambda_{\triple'} = 0$.  
    
    Finally, we show that $Obj(\bflambdahat,\bfmuhat) = Obj(\bflambdatilde,\bfmutilde)$. This follows from:
\begin{subequations}
\begin{align}
    & Obj(\bflambdahat,\bfmuhat) = -\sum\limits_{\triple \in \tripleSet} (b^T\bflambdahat_\triple +c^T\bflambdahat_\triple \vhat_{\triple}) - \sum\limits_{\graphnode \in \graphnodes} \muhat_\graphnode \label{obj.1} \\
    =&\; -\sum\limits_{\triple \in \tripleSet \setminus \{\triple'\}} (b^T\bflambdahat_\triple +c^T\bflambdahat_\triple \vhat_{\triple}) - 
    (b^T\bflambdahat_{\triple'} +c^T\bflambdahat_{\triple'} \vhat_{\triple'}) 
    - \sum\limits_{\graphnode \in \graphnodes \setminus \{\graphnode',\graphnode'',\graphnode'''\}} \muhat_\graphnode 
    - \muhat_{\graphnode'} - \muhat_{\graphnode''} - \muhat_{\graphnode'''} \label{obj.2} \\
    =&\; -\sum\limits_{\triple \in \tripleSet \setminus \{\triple'\}} (b^T\bflambdatilde_\triple +c^T\bflambdatilde_\triple \vhat_{\triple}) - 0 
    - \sum\limits_{\graphnode \in \graphnodes \setminus \{\graphnode',\graphnode'',\graphnode'''\}} \mutilde_\graphnode 
    - \mutilde_{\graphnode'} - \mutilde_{\graphnode''} - \mutilde_{\graphnode'''} - \lambdatilde_{\triple',1} - \lambdatilde_{\triple',2} - 2\lambdatilde_{\triple',3} \label{obj.3} \\
    =&\; -\sum\limits_{\triple \in \tripleSet \setminus \{\triple'\}} (b^T\bflambdatilde_\triple +c^T\bflambdatilde_\triple \vhat_{\triple}) 
    - \sum\limits_{\graphnode \in \graphnodes} \mutilde_\graphnode -b^T\bflambdatilde_{\triple'} = Obj(\bflambdatilde,\bfmutilde) \label{obj.4} 
\end{align}
\end{subequations}
Equality~\eqref{obj.2} follows by splitting the first sum in~$\triple'$ and the second sum in 
$\{\graphnode',\graphnode'',\graphnode'''\}$.
The equality in~\eqref{obj.3} follows by substituting~\eqref{defmultiplierhat}.  The equality in~\eqref{obj.4} follows from the definition of $b$ in~\eqref{lp.con1}. The final equality follows by noting that $b^T\bflambdatilde_{\triple'} = b^T\bflambdatilde_{\triple'} + c^T\bflambdatilde_{\triple'}v_{\triple'}$ since $v_{\triple'} = 0$ and collecting the terms into the summation over $\triple$ in $\tripleSet \setminus \{\triple'\}$.  
\end{proof}

\begin{lemma}\label{lemma:bndlambda3}
Let $(\optbflambda{\bfvhat},\optbfmu{\bfvhat})$ be an optimal solution to~\eqref{dual} 
as stated in Lemma~\ref{lemma:complSlackness}. 
Then 
$\optlambda{\bfvhat}_{\triple,3}, \optmu{\bfvhat}_\graphnode \leq \eta$.
\end{lemma}
\begin{proof}
Consider the term involving $\optbflambda{\bfvhat}_{\triple}$ 
in~\eqref{dual.obj} for some triple~$\triple$. This can be simplified as
\begin{equation}
\begin{aligned}
    \textbf{b}^T\optbflambda{\bfvhat}_{\triple} + \textbf{c}^T \optbflambda{\bfvhat}_{\triple} \vhat_\triple
    = \left\{ \begin{aligned}
        \optlambda{\bfvhat}_{\triple,3} & \text{ if } \vhat_\triple = 1 \\
        \textbf{b}^T\optbflambda{\bfvhat}_\triple = 0 & \text{ if } \vhat_\triple = 0    
    \end{aligned} \right.
\end{aligned}\label{simplifydualobj}
\end{equation}
which follows by substituting for $\textbf{b},\textbf{c}$ from~\eqref{lp.con1} and  from Lemma~\ref{lemma:complSlackness}.
Thus, the optimal value of the objective in~\eqref{dual.obj} can be reduced to 
\begin{equation}
Obj(\bflambdahat,\bfmuhat) =     -\sum_{\triple \in \tripleSet} (\textbf{b}^T\optbflambda{\bfvhat}_\triple +\textbf{c}^T\optbflambda{\bfvhat}_\triple \vhat_{\triple}) - \sum\limits_{\graphnode \in \graphnodes} \optmu{\bfvhat}_\graphnode 
    = - \sum_{\triple \in \tripleSet : \vhat_\triple = 1} \optlambda{\bfvhat}_{\triple,3} - \sum_{\graphnode \in \graphnodes} \optmu{\bfvhat}_\graphnode 
    \geq -\eta,
    \label{simplifydualobj2}
\end{equation}
where the first equality follows from~\eqref{simplifydualobj} and the inequality from Lemma~\ref{lemma:bnLDualobj}. Combining $\optbflambda{\bfvhat}, \optbfmu{\bfvhat} \geq 0$ with~\eqref{simplifydualobj2} yields that $\optlambda{\bfvhat}_{\triple,3} \leq \eta$ for all $\triple$ in $\tripleSet$ such that $\vhat_\triple = 1$ and $\optmu{\bfvhat}_\graphnode \leq \eta$ for all $\graphnode$ in $\graphnodes$. To complete the proof it suffices to recall that, by Lemma~\ref{lemma:complSlackness},  $\optlambda{\bfvhat}_{\triple,3} = 0 \leq \eta$ for all $\triple$ in $\tripleSet$ such that $\vhat_\triple = 0$.
\end{proof}

Lemma~\ref{lemma:bndlambda3} yields that $M_{\triple,3} = \eta$ for all $\triple$ in $\tripleSet$ and $M_{\graphnode} = \eta$ for all $\graphnode$ in $\graphnodes$. Next, we show that 
the bounds for~$\lambda^{\bfvhat}_{\triple,1}$ and~$\lambda^{\bfvhat}_{\triple,2}$ 
are also finite for all $\triple$ in $\tripleSet$.

\begin{lemma}\label{lemma:bndlambda12}
Let $(\optbflambda{\bfvhat},\optbfmu{\bfvhat})$ be an optimal solution to~\eqref{dual} as stated in Lemma~\ref{lemma:complSlackness}. 
There exists a finite $M_{\triple,j}$ for each $\triple \in \tripleSet$ and  $j = 1,2$ such that $\optlambda{\bfvhat}_{\triple,j} \leq M_{\triple,j}$.
\end{lemma}
\begin{proof}
Consider the inequality in~\eqref{dual.con1} for $\graphnode \in \graphnodes$. This can be rewritten for $(\optbflambda{\bfvhat},\optbfmu{\bfvhat})$ as
\begin{equation}
    \sum\limits_{\triple : \graphnode = \tailone(\triple)}
    \optlambda{\bfvhat}_{\triple,1} + \sum\limits_{\triple : \graphnode = \tailtwo(\triple)} \optlambda{\bfvhat}_{\triple,2} \leq \beta_\graphnode + 
    \sum\limits_{\triple : \graphnode = \tailone(\triple)}
    \optlambda{\bfvhat}_{\triple,3} +
    \sum\limits_{\triple : \graphnode = \tailtwo(\triple)}
    \optlambda{\bfvhat}_{\triple,3} + \sum\limits_{\triple : \graphnode = \head(\triple)} 
    (\optlambda{\bfvhat}_{\triple,1} + \optlambda{\bfvhat}_{\triple,2} - \optlambda{\bfvhat}_{\triple,3})
    + \optmu{\bfvhat}_\graphnode 
    \label{bndlambda12.1}
\end{equation}
From~\eqref{simplifydualobj2} 
we have that $\sum\limits_{\triple \in \tripleSet} \optlambda{\bfvhat}_{\triple,3} + \sum\limits_{\graphnode \in \graphnodes} \optmu{\bfvhat}_\graphnode \leq \eta$. Then we can upper bound the terms involving $\optlambda{\bfvhat}_{\triple,3}$ and $\optmu{\bfvhat}_\graphnode$ on the right hand side of~\eqref{bndlambda12.1} as
\begin{equation}
    \sum\limits_{\triple : \graphnode = \tailone(\triple)}
    \optlambda{\bfvhat}_{\triple,3} +
    \sum\limits_{\triple : \graphnode = \tailtwo(\triple)}
    \optlambda{\bfvhat}_{\triple,3} + \optmu{\bfvhat}_\graphnode     
    \leq 
    \sum\limits_{\triple : \tripleSet}
    \optlambda{\bfvhat}_{\triple,3} + \sum_{\graphnode \in \graphnodes} \optmu{\bfvhat}_\graphnode     
    \leq \eta 
    \label{bndlambda12.2}
\end{equation}
where the first inequality follows by noting that either $\graphnode = \tailone(\triple)$ or $\graphnode = \tailtwo(\triple)$ but not both, and from the non-negativity of multipliers. The second inequality follows from~\eqref{simplifydualobj2} and Lemma~\ref{lemma:complSlackness}.  Thus, the inequality~\eqref{bndlambda12.1} can be simplified to 
\begin{equation}
    \sum\limits_{\triple : \graphnode = \tailone(\triple)}
    \optlambda{\bfvhat}_{\triple,1} + \sum\limits_{\triple : \graphnode = \tailtwo(\triple)} \optlambda{\bfvhat}_{\triple,2} \leq 
    \beta_\graphnode + 
    \eta + \sum\limits_{\triple : \graphnode = \head(\triple)} 
    (\optlambda{\bfvhat}_{\triple,1} + \optlambda{\bfvhat}_{\triple,2} - \optlambda{\bfvhat}_{\triple,3}) \leq \beta_\graphnode + 
    \eta + \sum\limits_{\triple : \graphnode = \head(\triple)} 
    (\optlambda{\bfvhat}_{\triple,1} + \optlambda{\bfvhat}_{\triple,2})    
    \label{bndlambda12.3}
\end{equation}
where the first inequality follows from~\eqref{bndlambda12.2} and the second inequality from the nonnegativity of $\optlambda{\bfvhat}_{\triple,3}$. Observe that the right hand side of~\eqref{bndlambda12.3} involves the multipliers $\optlambda{\bfvhat}_{\triple,1}$ and $\optlambda{\bfvhat}_{\triple,2}$ for all $\triple$ such that $\graphnode = \head(\triple)$ i.e., the triples~$\triple$ for which $\graphnode$ is the head. If an upper bound is available for such multipliers then we can use~\eqref{bndlambda12.3} to derive an upper bound on the arcs in which~$\graphnode$ is a tail.  

We 
show by induction that~$M_{\triple,1}$ and~$M_{\triple,2}$ are finite; the argument delivers an iterative procedure to construct these bounds. First, consider $\mathcal{J}_1 \coloneqq \{ \graphnode \in \graphnodes \,|\, |\graphnode| = 1\}$.  We have $\{\triple \in \tripleSet \,|\, \graphnode = \head(\triple)\} = \emptyset$ for each~$\graphnode$ in~$\mathcal{J}_1$, i.e., $\graphnode$ cannot be the head of any triple. Therefore, the inequality~\eqref{bndlambda12.3} for~$\indexset \in \mathcal{J}_1$
becomes 
    \begin{equation}
        \sum\limits_{\triple : \graphnode = \tailone(\triple)}
        \optlambda{\bfvhat}_{\triple,1} + \sum\limits_{\triple : \graphnode = \tailtwo(\triple)} \optlambda{\bfvhat}_{\triple,2} \leq \beta_\graphnode + \eta.
        \label{bndlambda12.4}
    \end{equation}
    Therefore, $M_{\triple,1} \leq \beta_\graphnode + \eta$ and
    $M_{\triple,2} \leq \beta_\graphnode + \eta$ for each $\triple$ in $\tripleSet$ such that $\tailone(\triple) \in \mathcal{J}_1$ or $\tailtwo(\triple) \in \mathcal{J}_1$, respectively.  Next, we consider $\mathcal{J}_2 \coloneqq \{ \graphnode \in \graphnodes \,|\, |\graphnode| = 2\}$. For each~$J$ in $\mathcal{J}_2$, any~$\triple$ in $\{\triple \in \tripleSet \,|\, \graphnode = \head(\triple)\}$ is such that 
    $|\tailone(\triple)| = 1$ and $|\tailtwo(\triple)| = 1$. Therefore, upper bounds $M_{\triple,1}$ and $M_{\triple,2}$ have been identified for~$\optlambda{\bfvhat}_{\triple,1}$ and~$\optlambda{\bfvhat}_{\triple,2}$, respectively, in the first iteration. Hence~\eqref{bndlambda12.3} can be written  as 
    \begin{equation}
        \sum\limits_{\triple : \graphnode = \tailone(\triple)}
        \optlambda{\bfvhat}_{\triple,1} + \sum\limits_{\triple : \graphnode = \tailtwo(\triple)} \optlambda{\bfvhat}_{\triple,2} \leq 
        \beta_\graphnode + \eta + \sum\limits_{\triple : \graphnode = \head(\triple)} 
        (M_{\triple,1} + M_{\triple,2}).   
        \label{bndlambda12.5}
    \end{equation}
    Thus $M_{\triple,1}$ for $\tailone(\triple) \in \mathcal{J}_2$ and $M_{\triple,2}$ for $\tailtwo(\triple) \in \mathcal{J}_2$ can be obtained from the right hand side of~\eqref{bndlambda12.5}.  We can repeat the above for $\mathcal{J}_k \coloneqq \{ \graphnode \in \graphnodes \,|\, |\graphnode| = k\}$, $3 \leq k \leq \nvariables$, by considering sets of increasing cardinality to determine all the bounds $M_{\triple,j}$ for $j = 1,2$.  
\end{proof}

\subsection{Best Bound MIP}

Let~$\bfsetV$ denote the set of vectors~$\bfv$ in $\B^{|\tripleSet|}$ composing a feasible solution to~\eqref{minlin}, and let
$\bfsetV_k \coloneqq \{ \bfv  \,|\, \bfv \in \bfsetV, \|\bfv\|_1 \leq k\}$, i.e., the elements of~$\bfsetV_K$ represent the proper triple sets containing at most~$k$ elements. 
We consider the following bilevel formulation to identify an element of~$\bfsetV_k$
that yields the best LP relaxation bound.
\begin{subequations}
\begin{align}
    \max\limits_{\bfv \in \bfsetV_k} \min\limits_{\bfy \in [0,1]^{|\graphnodes|}} 
    &\; \sum\limits_{i = 1}^{\nmonomials} \alpha_i y_{\indexset_i} &\, \label{maxminform.obj} \\
    \text{s.t.} &\; B\bfy_\triple \leq \textbf{b} + \textbf{c} v_\triple 
    &\,\forall\, \triple = (\graphnode,\graphnode',\graphnode'') \in \tripleSet \label{maxminform.con1}
\end{align}\label{maxminform}
\end{subequations}
Note that $\bfy$ and $\bfv$ variables are defined over~$\tripleSet$ and~$\graphnodes$, 
respectively, as in~\eqref{lp}. Further, we use~$\bfy_\triple$ to denote the collection of variables $(y_{\graphnode},y_{\graphnode'},y_{\graphnode''})$, where~$\triple = (\indexset,\indexset',\indexset'')$. We use strong duality to cast 
\eqref{maxminform} as a single-level maximization MIP.

\begin{theorem}\label{lemma:mipbb}
The max-min problem in~\eqref{maxminform} is equivalent to the following MIP:
%
\begin{eqnarray}
    \max\limits_{\bfv \in \bfsetV_K, \bflambda \in \R^{|\tripleSet|}, \bfmu \in \R^{|\graphnodes|}} 
    &\; -\sum\limits_{\triple \in \tripleSet} \lambda_{\triple,3} - \sum\limits_{\graphnode \in \graphnodes} \mu_\graphnode &\, \label{maxmipform.obj} \\
    \text{s.t.} 
    &\; \eqref{dual.con1}-\eqref{dual.con3} \label{maxmipform.con1} \\
    &\; \lambda_{\triple,j} \leq M_{\triple,j} v_\triple &\, \triple \in \tripleSet, j = 1,2,3. \label{maxmipform.con2} 
\end{eqnarray}\label{maxmipform}
\end{theorem}
\begin{proof} First, we show 
that~\eqref{maxminform} can be cast as the following single-level maximization problem:
\begin{subequations}
\begin{align}
    \max\limits_{\bfv \in \bfsetV_k, \bflambda \in \R^{|\tripleSet|}, \bfmu \in \R^{|\graphnodes|}} 
    &\; -\sum\limits_{\triple \in \tripleSet} (\textbf{b}^T\bflambda_\triple +\textbf{c}^T\bflambda_\triple v_{\triple}) - \sum\limits_{\graphnode \in \graphnodes} \mu_\graphnode &\, \label{maxform.obj} \\
    \text{s.t.} 
    &\; \eqref{dual.con1} - \eqref{dual.con3} \label{maxform.con1}
\end{align}\label{maxform}
\end{subequations}
From Lemma~\ref{lemma:LPisBnLD} we have that the inner minimization problem in~\eqref{maxminform}, given by~\eqref{lp},  attains a finite optimal value for any $\bfv$. By strong duality of LP, the optimal objective value of the inner minimization problem is equal to the optimal objective value of the dual~\eqref{dual}. 
Substituting~\eqref{lp} by~\eqref{dual}
and noting that $\max\limits_{\bfv \in \bfsetV_K}$ and $\max\limits_{\bflambda,\bfmu}$ can be combined into a single level proves the claim.
%
%

Formulation~\eqref{maxform} has linear constraints but a bilinear objective, since~$v_\triple$ multiplies~$\textbf{c}^T\lambda_\triple$. 
By Lemma~\ref{lemma:complSlackness}, the optimal solution to~\eqref{dual} satisfies $\optbflambda{\bfv}_\triple = 0$ for each~$\triple$ such that $v_\triple = 0$. Lemmas~\ref{lemma:bndlambda3} and~\ref{lemma:bndlambda12} provide upper bounds on the optimal multipliers~$(\optbflambda{\bfvhat},\optbfmu{\bfvhat})$. Hence, the constraints~\eqref{maxmipform.con2} are valid. Finally, the simplification of the objective function follows from~\eqref{simplifydualobj2}, so the result follows. 
\end{proof}

\section{Computational results}
\label{sec:computational_results}

In this section, we present the results of a numerical study conducted to demonstrate the performance of the algorithms introduced in this paper. The baseline algorithm is~\texttt{Seq}, the sequential approach adopted by state-of-the-art global optimization solvers, in which the variables in each term are arbitrarily ordered and products of variables are linearized sequentially (throughout all the monomials in which they occur). We evaluate the performance of three algorithms to solve the~\polyopt{}:  \texttt{MinLin}, the MIP formulation presented in~\S\ref{sec:exact}; \texttt{Greedy}, the sub-optimal algorithm for the minimum size~\recmcclin{} presented in~\S\ref{sec:heuristic}; and \texttt{BB}, the best bound MIP of~\S\ref{sec:BB}. In some experiments, we use \texttt{Full} to refer to the linearization containing all the triples.

We solve each instance of our data set by identifying a proper triple set~$T$ first and then use~$T$ to solve the optimization problem. The time limit allotted to all these operations is 10 minutes. We set a time limit of 30 seconds for each execution of~\texttt{MinLin} and~\texttt{BB}; the runtimes of~\texttt{Seq} and~\texttt{Greedy} are negligible. The linearization of~\texttt{Greedy} is given as warm start to~\texttt{MinLin}, and the best linearization~$T$ identified by~\texttt{MinLin} is used as warm start for~\texttt{BB}. Additionally, the size of~$T$ is given as cardinality constraint for~\texttt{BB} ($K = |T|$). 

We implement and run our experiments using  Python 3.9 on a 4.20 GHz Intel(R) Core(TM) i7-7700K processor with a single thread and 32GB of RAM. We use \textbf{Gurobi} 9.1 (\cite{gurobi}) to solve all the optimization problems. For the second step (solution of the original problem based on the linearization identified in the first step), we deactivate the generation of cuts by setting the parameter \texttt{Cuts} to zero.

\subsection{Instances}
\label{sec:test_set}

We use four families of instances in our experiments. The first benchmark contains instances of multilinear optimization problems introduced by~\cite{dpks:20} (see also~\cite{bao2015global});  The second is a traditional benchmark data set used in computer vision~(\cite{crh:17}). Finally, the \texttt{autocorr} instances were extracted from {POLIP}, a library of polynomially constrained mixed-integer programming, (\url{http://polip.zib.de}). 


\paragraph{Multilinear optimization problems:} The first data set consists of 330 unconstrained multilinear problems, which is divided into two categories: \texttt{mult3} and~\texttt{mult4}. For each combination of~$n \in \{20,25,30,35,40\}$ and $m \in \{50,60,\ldots,150\}$, we generate 3 instances in which all monomials are of degree~$3$ (for \texttt{mult3}) and 3 instances with monomials of degree~$4$ (for \texttt{mult4}). The variables in each monomial are chosen independently and uniformly at random (and without replacement). The coefficients of each monomial are integer values chosen uniformly from the interval~$[-100,100]$.



\paragraph{Vision Instances}

The~\texttt{vision} instances model an image restoration problem, which has been widely studied in computer vision (see e.g., \cite{crh:17}).  
The problem can be modeled as a \polyopt{} $f(\bfx) = L(\bfx) + H(\bfx)$, with $L(\bfx)$ being an affine function and $H(\bfx)$ a multilinear function of degree four. In our experiments, we use the 45 instances generated in~\cite{crh:17}, for which $n \in \{100, 150, 225\}$. The instances of a given size share the same multilinear function $H(\bfx)$, i.e., they only differ  in the coefficients of~$L(\bfx)$.

\paragraph{Auto-correlation Instances}

The \texttt{autocorr} instances are from POLIP (\url{http://polip.zib.de}). We consider instances with~$n \in \{20,25,30,35,40,45\}$ in our experiments.

\subsection{Linearization Size}
\label{sec:experiments_linearization_size}

Figure~\ref{fig:linearizationsize} shows 
by how much~\texttt{MinLin} and \texttt{Greedy} change the size of the linearizations in comparison with~\texttt{Seq} for the \texttt{mult3} and \texttt{mult4} instances. All plots are cumulative and show the proportion of instances (in the~$x$-axis) achieving a reduction that is at least as large as the value indicated in the~$y$-axis. 
\begin{figure}[ht!]
\centering
\subfloat[][\centering \texttt{mult3} instances \label{fig:size_mult3}]{%
  \includegraphics[scale=0.35]{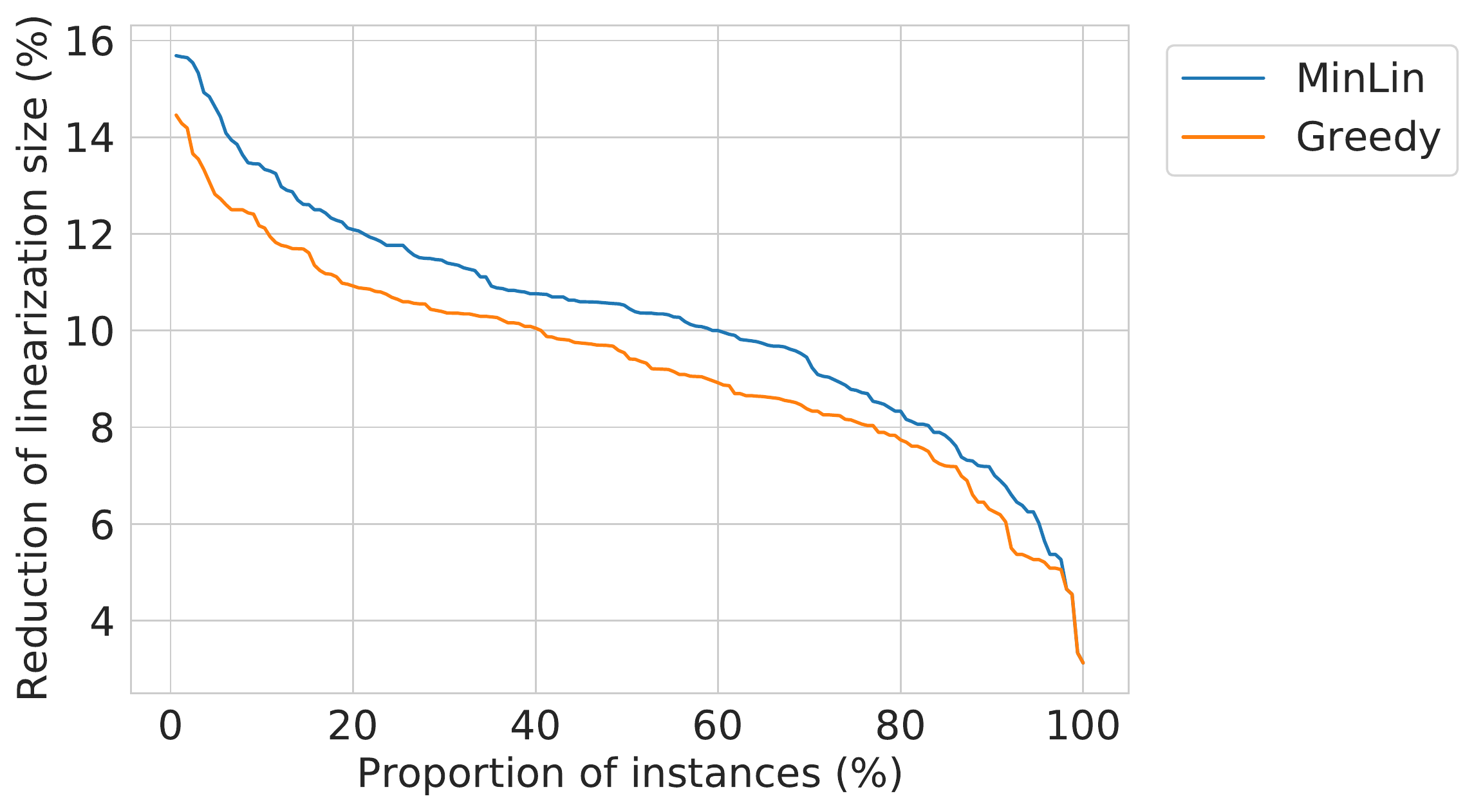}%
}
\subfloat[][\centering  \texttt{mult4} instances \label{fig:size_mult4}]{%
  \includegraphics[scale=0.35]{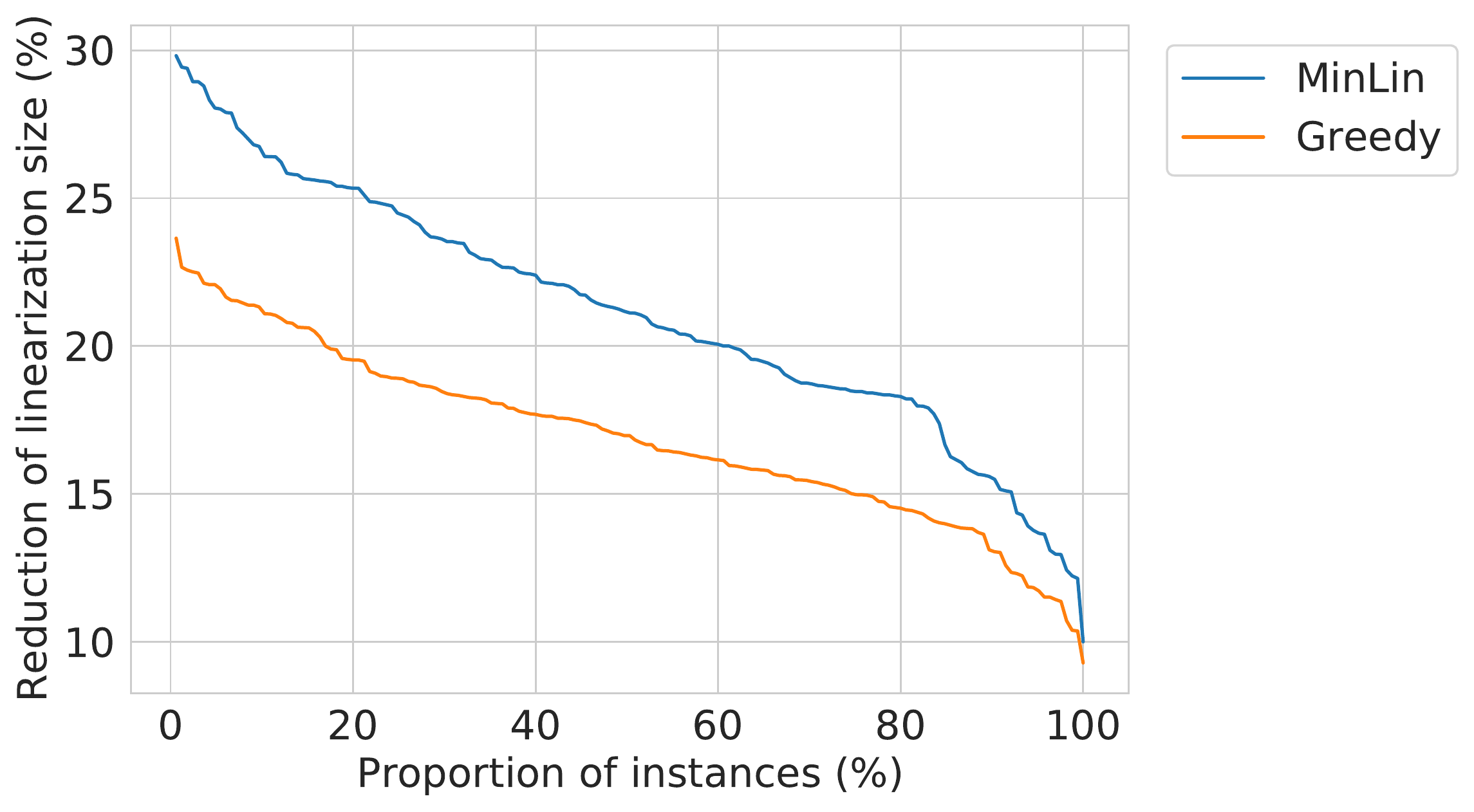}%
}
\caption{Reduction in the number of variables in comparison with \texttt{Seq} categorized by data set.}
\label{fig:linearizationsize}
\end{figure}
The results 
show that both \texttt{MinLin} and \texttt{Greedy} identify linearizations that are significantly smaller than the linearizations of \texttt{Seq}. Moreover, \texttt{MinLin} is consistently better than \texttt{Greedy}, with more pronounced differences in the~\texttt{mult4} instances.
In contrast, the structure of the \texttt{vision} and \texttt{autocorr} instances lead to stable results, i.e., the impact of the dimensions of these instances on the relative performance of the algorithms is negligible, so we omit these plots. In the case of \texttt{vision}, \texttt{Seq} delivers minimum linearizations already, so \texttt{MinLin} brings no gains; in contrast, the linearizations produced by~\texttt{Greedy} have 15\% more variables. Finally, all algorithms deliver linearizations of the same size for all \texttt{autocorr} instances.

Figure~\ref{fig:linearizationtime} shows the performance profiles of~\texttt{MinLin} for each data set. Each plot is divided into two parts. On the left, we report the percentage of instances that were solved to optimality (in the $y$-axis) within the amount of time indicated in the $x$-axis; the largest value of~$x$ is 30 seconds, which is the time limit we set for~\texttt{MinLin}. On the right, we indicate the percentage of instances for which \textbf{Gurobi} obtained an optimality gap inferior to the value indicated in the~$x$-axis; we assume that instances solved to optimality have an optimality gap equal to 0, so the rightmost part of the plot is the natural extension of the leftmost part. 
\begin{figure}[ht!]
\centering
\subfloat[][\centering \texttt{mult3} instances \label{fig:linearizationtime_mult3}]{%
  \includegraphics[scale=0.35]{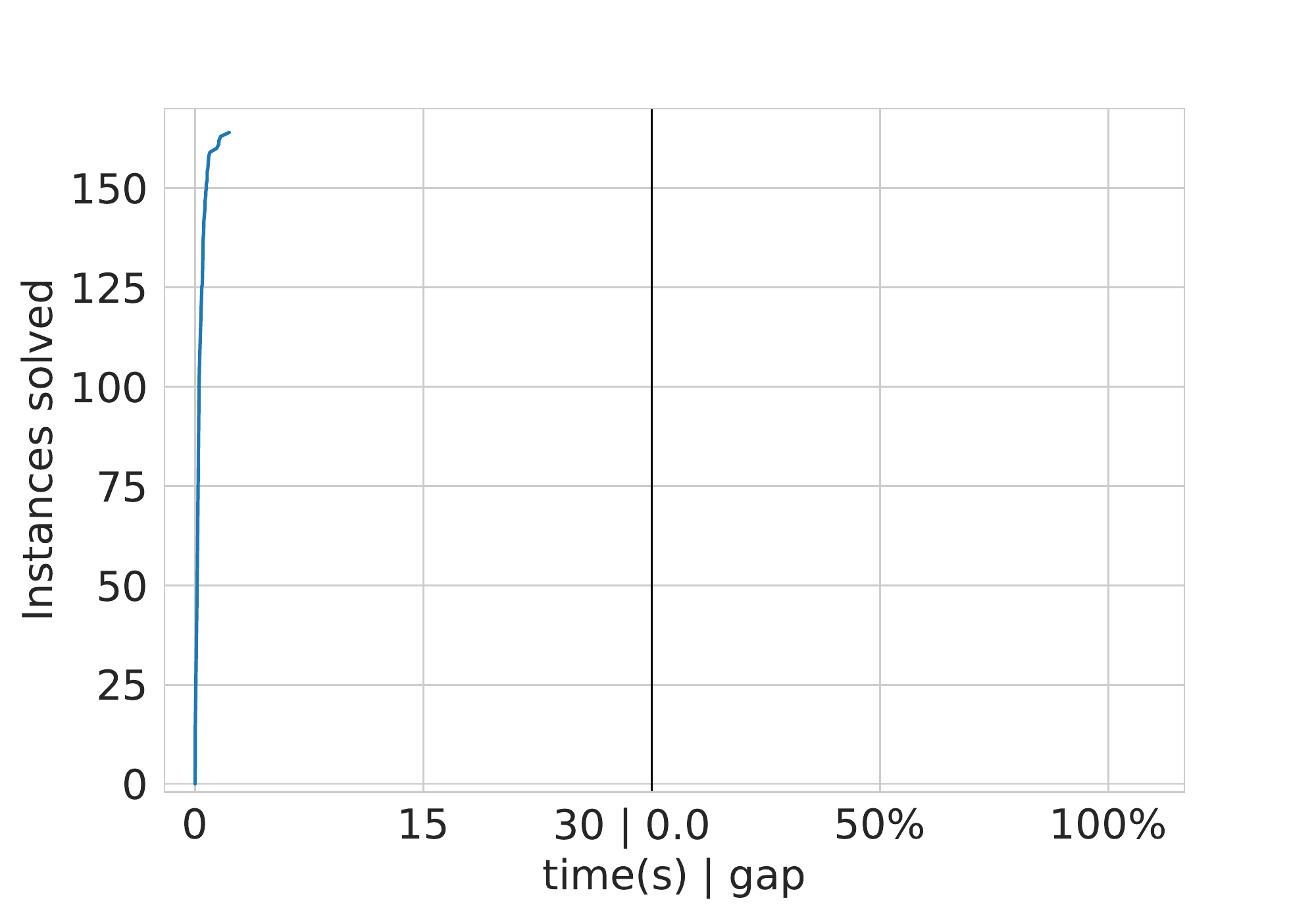}%
}
\subfloat[][\centering  \texttt{mult4} instances \label{fig:linearizationtime_mult4}]{%
  \includegraphics[scale=0.35]{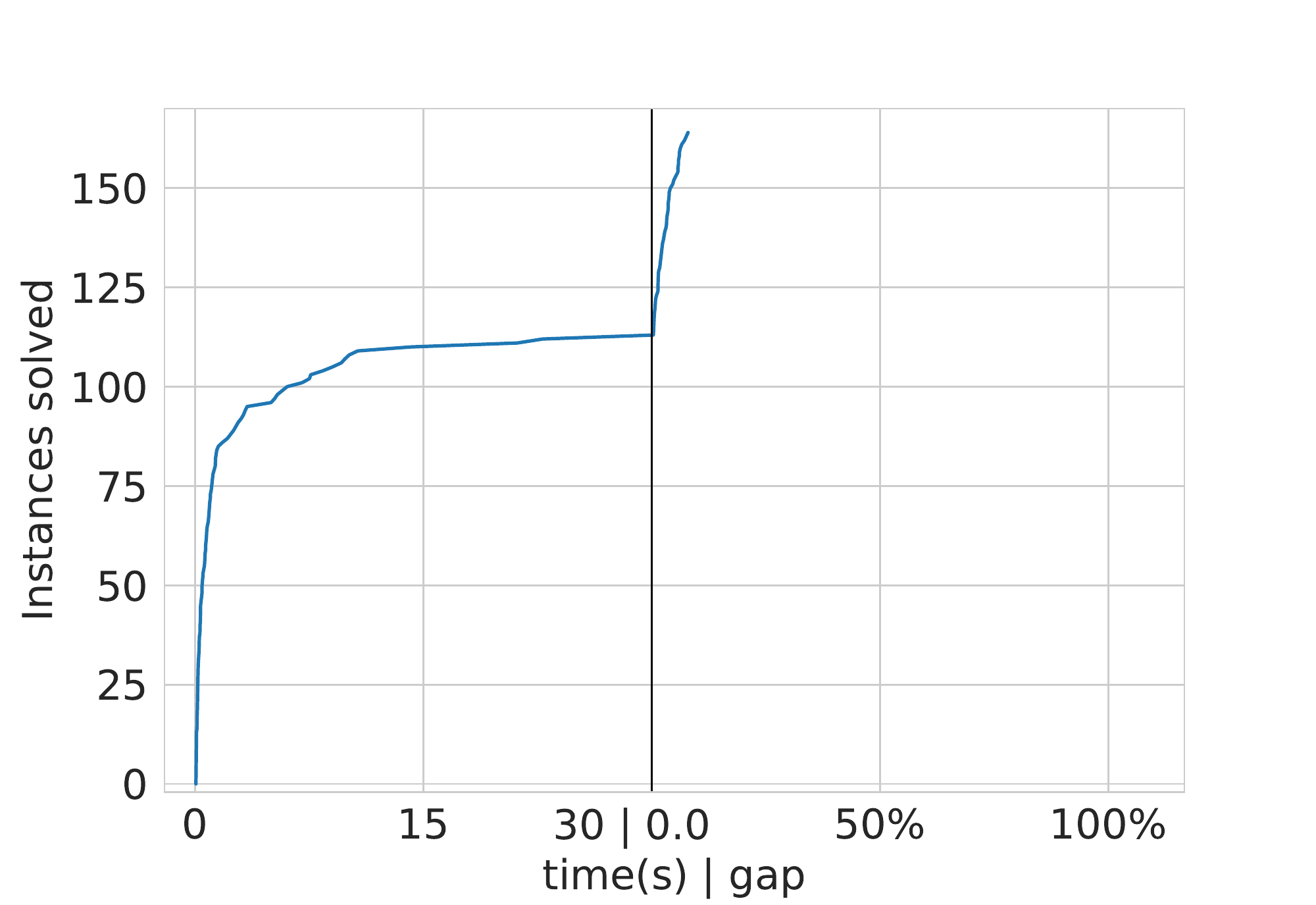}%
}\\
\subfloat[][\centering \texttt{vision} instances \label{fig:linearizationtime_vision}]{%
  \includegraphics[scale=0.34]{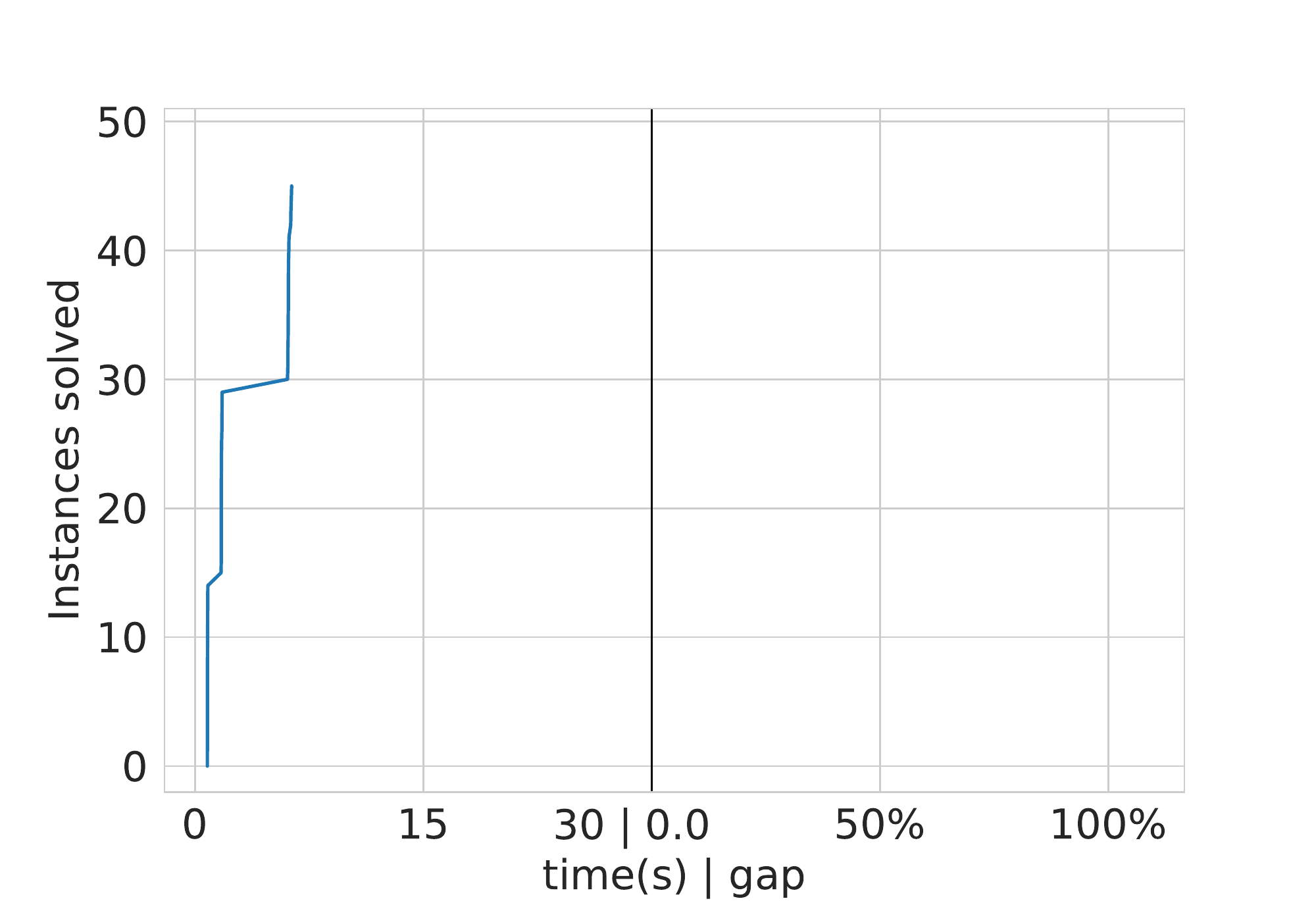}%
}
\subfloat[][\centering \texttt{autocorr} instances \label{fig:linearizationtime_autocorr}]{%
  \includegraphics[scale=0.34]{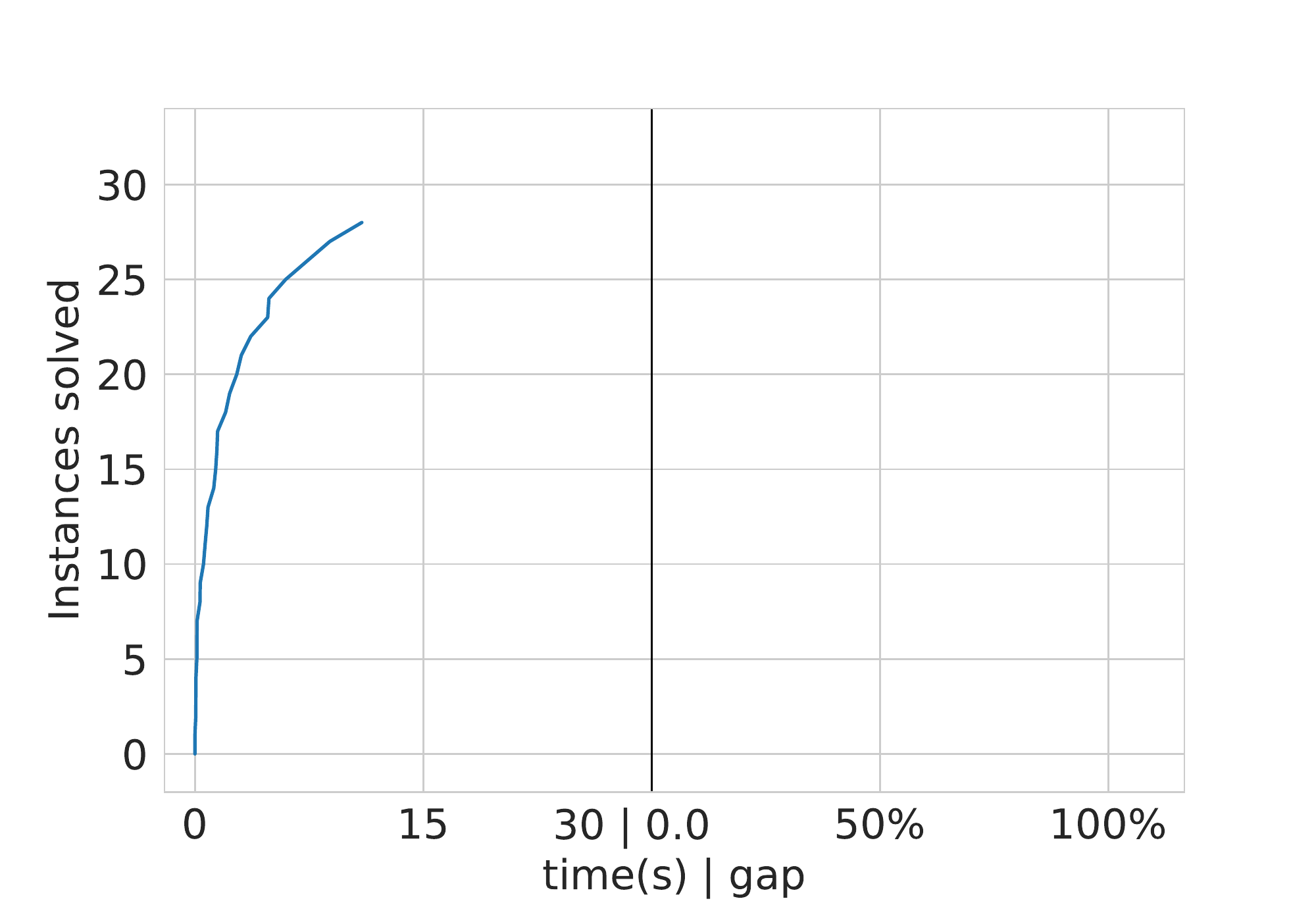}%
}
\caption{Cumulative performance plots of \texttt{MinLin} categorized by data set.}
\label{fig:linearizationtime}
\end{figure}
\texttt{MinLin} delivers strong performance and identifies a minimum linearization within less than 15 seconds for all instances in \texttt{mult3}, \texttt{vision}, and \texttt{autocorr}. In contrast, Figure~\ref{fig:size_mult4} shows that some instances of \texttt{mult4} cannot be solved to optimality within the time limit. Interestingly, we observe a correlation of 0.72 between the difference in the sizes of the linearizations produced by \texttt{Greedy} and \texttt{MinLin} and the runtime of \texttt{MinLin}, i.e., the harder instances benefit the most from an exact approach.


\subsection{Relaxation Bounds}
\label{sec:experiments_root_node_relaxations}

Next, we analyze the quality of the LP bounds of~\eqref{polyoptrml} corresponding to the minimum linearizations. More precisely, we calculate the  root-node (relaxation) gaps by comparing the LP bound 
of each algorithm \texttt{Alg} in $\{\texttt{Seq}, \texttt{Greedy},  \texttt{MinLin},  \texttt{BB}\}$ with the LP bound delivered by~\texttt{Full} 
as follows:
\begin{equation*}
\label{root_node_gap}
\text{Root-node gap} = \left(\dfrac{f_{\texttt{Full}} - f_{\texttt{Alg}}}{\max(|f_{\text{Full}}|, 10^{-3})}\right) \times 100,
\end{equation*}
where~$f_{\texttt{Alg}}$ is the root-node relaxation delivered by~\eqref{epolyopt} using the linearization of~$\texttt{Alg}$. Observe that~\texttt{Full} delivers the tightest 
formulation~\eqref{polyoptrml}, so~$f_{\texttt{Alg}} \leq f_{\texttt{Full}}$ holds for every instance. 

The results are presented in Figure~\ref{fig:root_node_gap}.
\begin{figure}[ht!]
\centering
\subfloat[][\centering \texttt{mult3} instances \label{fig:gap_d3}]{%
\includegraphics[scale=0.35]{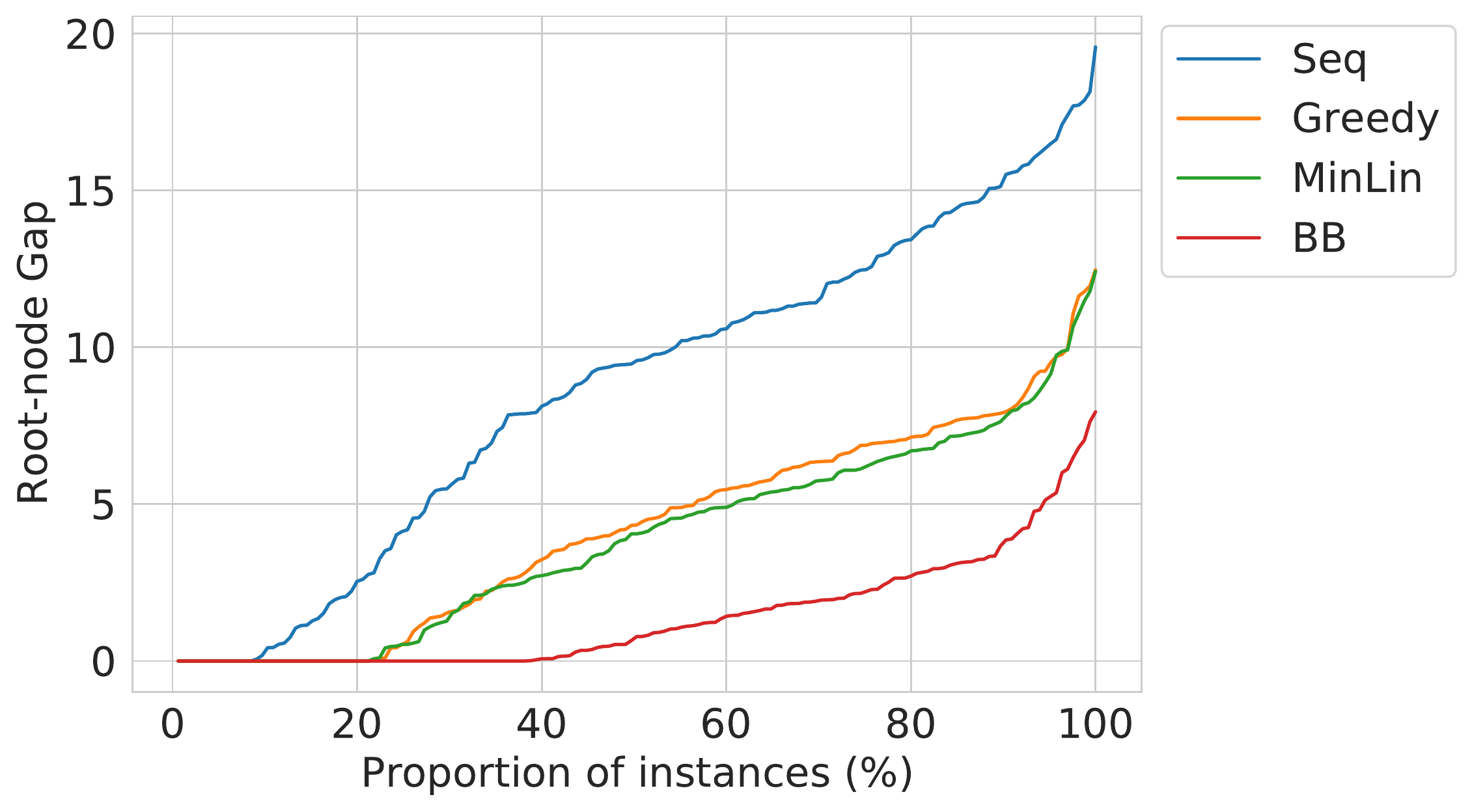}
}
\subfloat[][\centering \texttt{mult4} instances \label{fig:gap_d4}]{%
\includegraphics[scale=0.35]{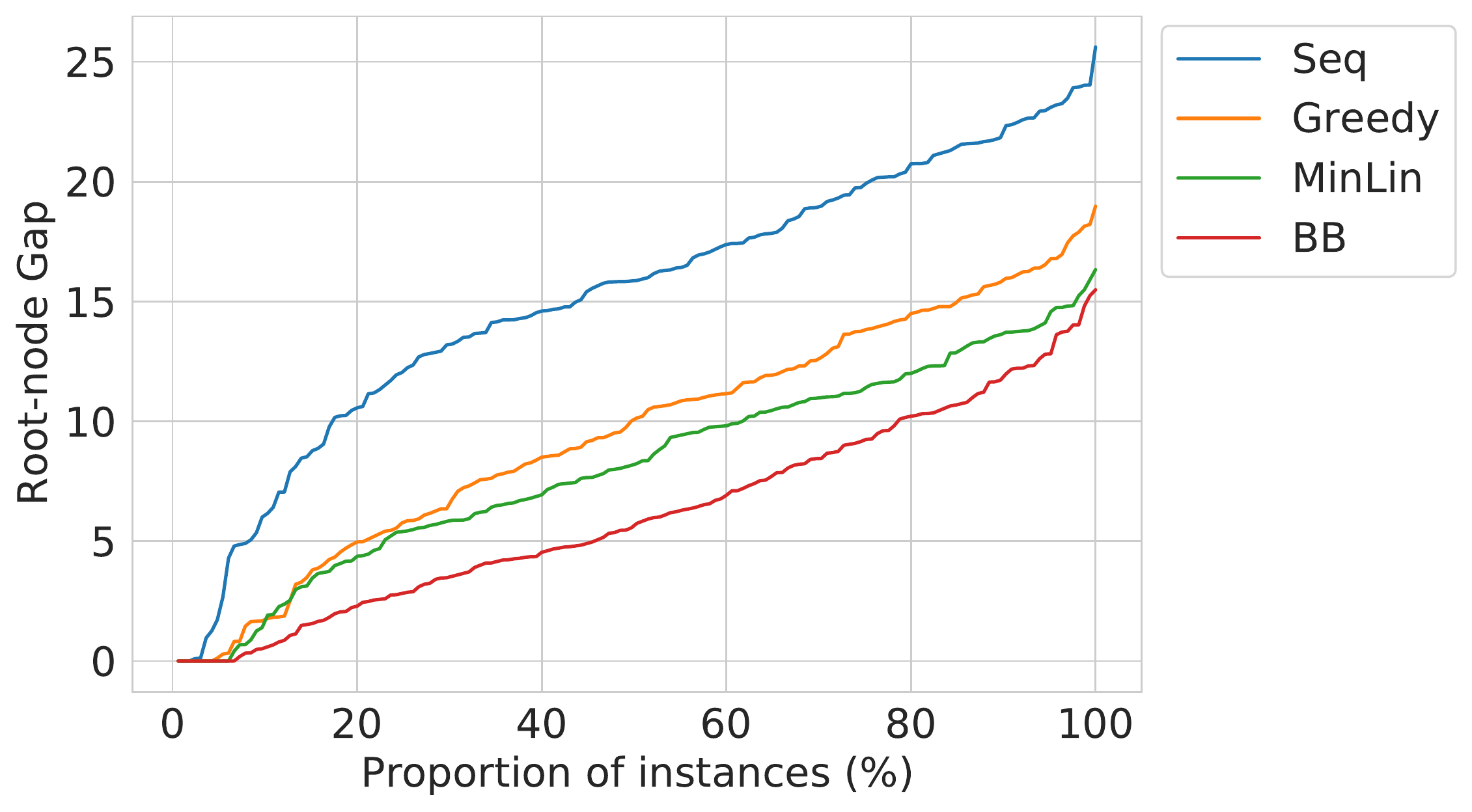}
}\\
\subfloat[][\centering \texttt{vision} instances \label{fig:gap_vision}]{%
\includegraphics[scale=0.35]{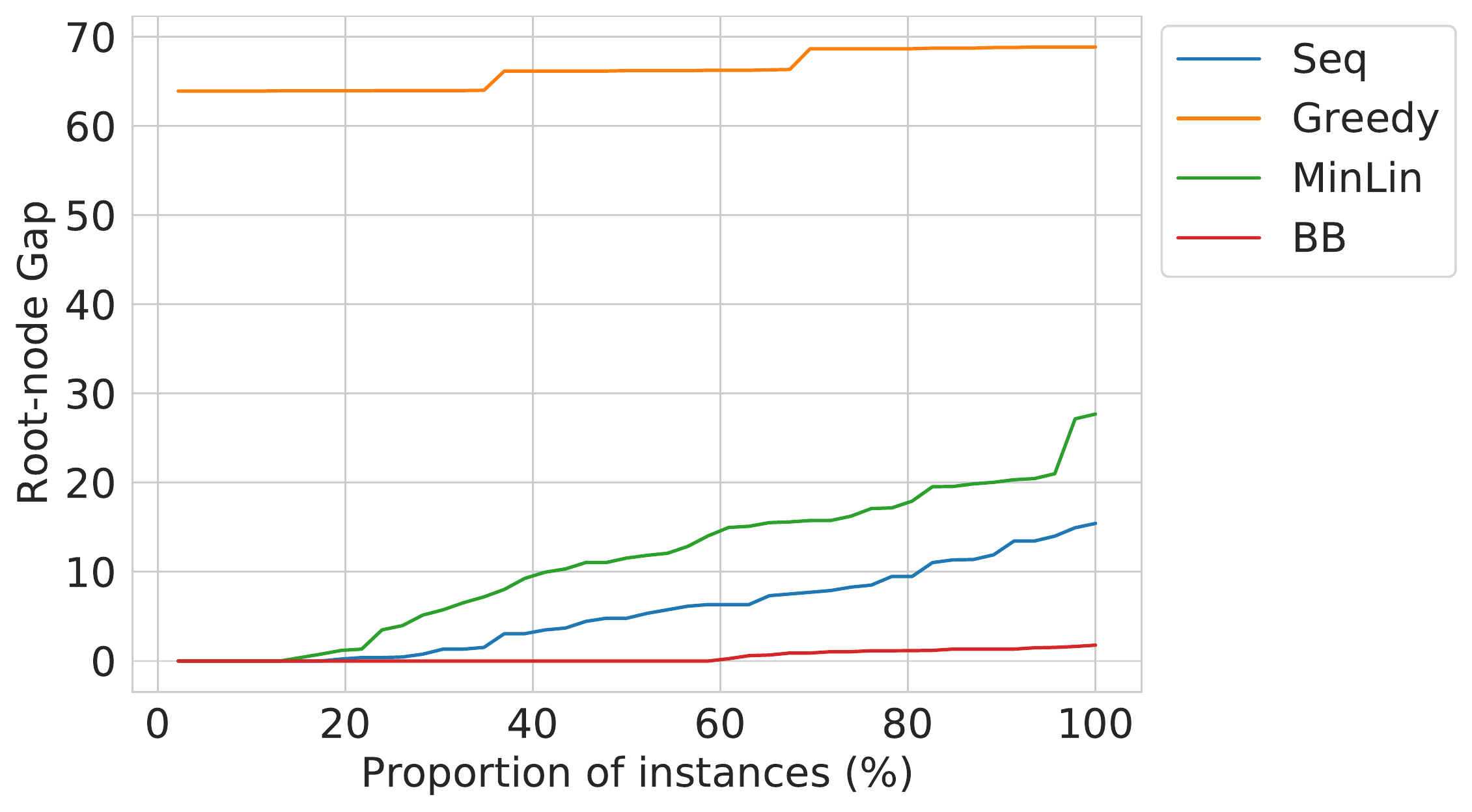}
}
\subfloat[][\centering \texttt{autocorr} instances \label{fig:gap_autocorr}]{%
\includegraphics[scale=0.35]{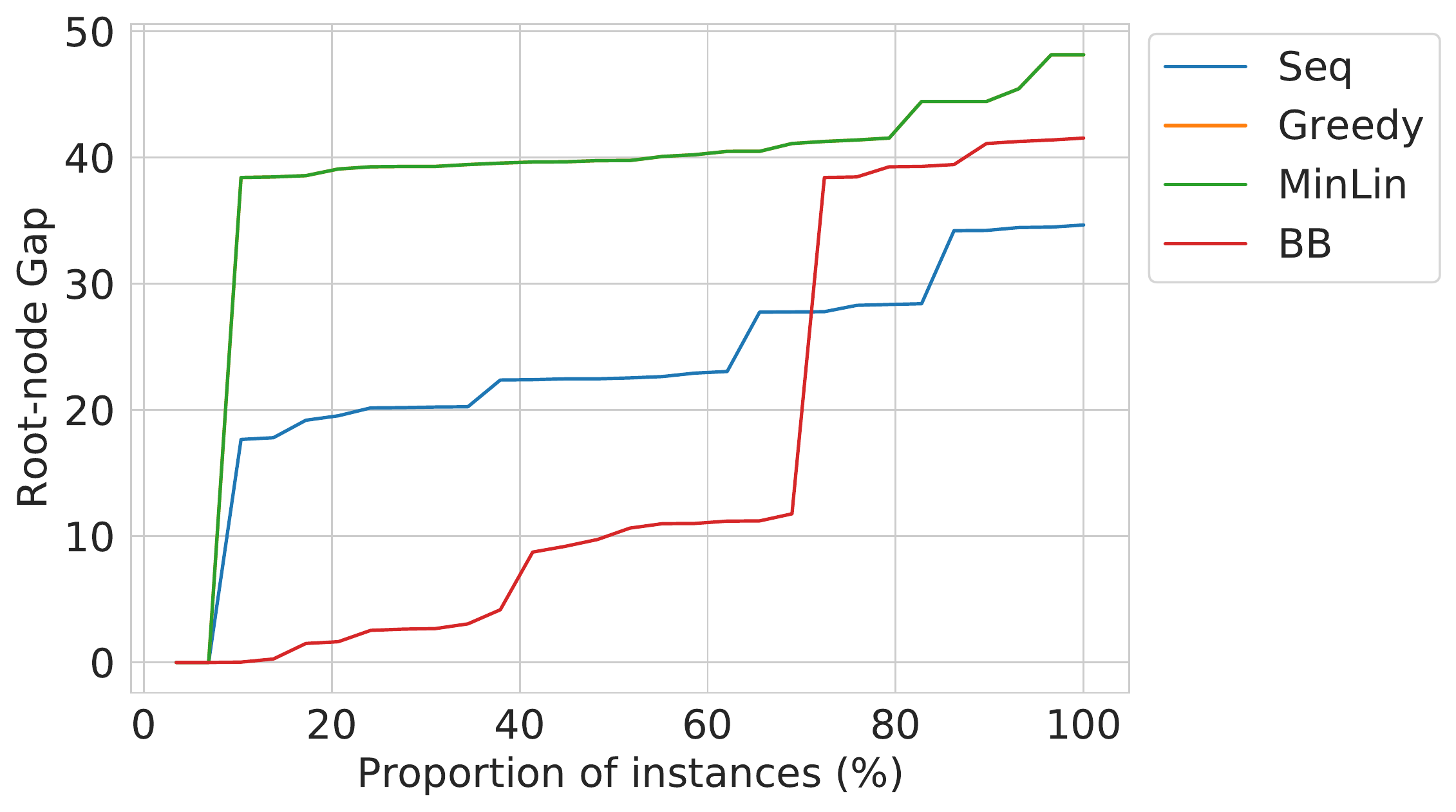}
}
\caption{Root-node Gap of all algorithms categorized by data set.}
\label{fig:root_node_gap}
\end{figure} 
The performances of the algorithms on the \texttt{mult3} and \texttt{mult4} instances are similar; \texttt{Seq} is the worst and \texttt{BB} is the best, whereas \texttt{Greedy} is slightly superior to~\texttt{MinLin}. The results for the \texttt{vision} instances are similar, with the remarkable exception of \texttt{Greedy}, which performs very poorly. These results show that \texttt{Greedy} is not only suboptimal with respect to the size of the linearization (see Example~\ref{ex:vision_greedy}), but may also deliver poor relaxation bounds. Finally, \texttt{Seq} and \texttt{Greedy} deliver exactly the same results for all instances in \texttt{autocorr}, which eventually is superior to both \texttt{MinLin} and \texttt{BB}.

\subsection{Experiments with global optimization solvers}
\label{sec:experiments_global_solvers}

Next, we report the results of our experiments for the entire optimization pipeline. Namely, we solve each instance by computing a set~$T$ of linearization triples first, using \texttt{Seq}, \texttt{Greedy}, \texttt{MinLin}, \texttt{BB}, or \texttt{Full}, and then we solve the following quadratically-constrained program (QCP) using~$T$. 
\begin{equation}
\begin{aligned}
   & \min && \sum_{i \in [\nmonomials]} \alpha_i y_{\indexset_i} \\
    & \text{s.t.} && y_{\indexset_1 \cup \indexset_2} = 
    y_{\indexset_1} y_{\indexset_2}, && \forall \left( {\indexset_1},
    {\indexset_2}, {\indexset_1 \cup \indexset_2}  \right)  \in T  \\
    &&& y_{\indexset} \in [0,1]^\nvariables, && \forall \indexset \in \graphnodes.
\end{aligned} \label{qcp}
\end{equation}

We use \textbf{Gurobi} to solve~\eqref{qcp}, so the QCP is obtained from~\eqref{polyopt} given a triple set~$T$ through the application of a McCormick linearization for each triple in~$T$ as a pre-processing operation. The runtime is limited at 10 minutes, from which we deduct the time spent to find a minimum linearization (in the case of \texttt{MinLin}) and a best-bound linearization (in the case of \texttt{MinLin} and \texttt{BB}). We do not consider the time spent on the construction of the models. We report the optimality gaps using the same expression adopted by \textbf{Gurobi}, i.e.,  we use 
\begin{equation*}
\text{Gap} = \left(\dfrac{f_{\texttt{ub}} - f_{\texttt{lb}}}{\max(|f_{\text{ub}}|, 10^{-3})}\right) \times 100,
\end{equation*}
where~$f_{\texttt{ub}}$ and $f_{\texttt{lb}}$ are the best upper and lower bound, respectively, obtained  within the time limit. 

Figure~\ref{fig:profiles_time_gap1} and~\ref{fig:profiles_time_gap2}
present the 
performance profiles of all algorithms. These plots are similar to those in Figure~\ref{fig:linearizationtime}, but we tailor the scales and the presentation for each data set in order to better exhibit the most relevant information.

Figure~\ref{fig:profiles_time_gap1} shows the results for \texttt{mult3} and \texttt{mult4}. The performance of the algorithms on these instances have an extreme behavior; they are either solved  to optimality within the time limit or no meaningful (i.e., finite) gap is identified. Therefore, we restrict the performance profile only to the runtime part; moreover, both coordinates are presented in log-scale.
\begin{figure}[htp]
\centering
\subfloat[][\centering \texttt{mult3} \label{fig:qcp_d3}]{%
  \includegraphics[scale=0.35]{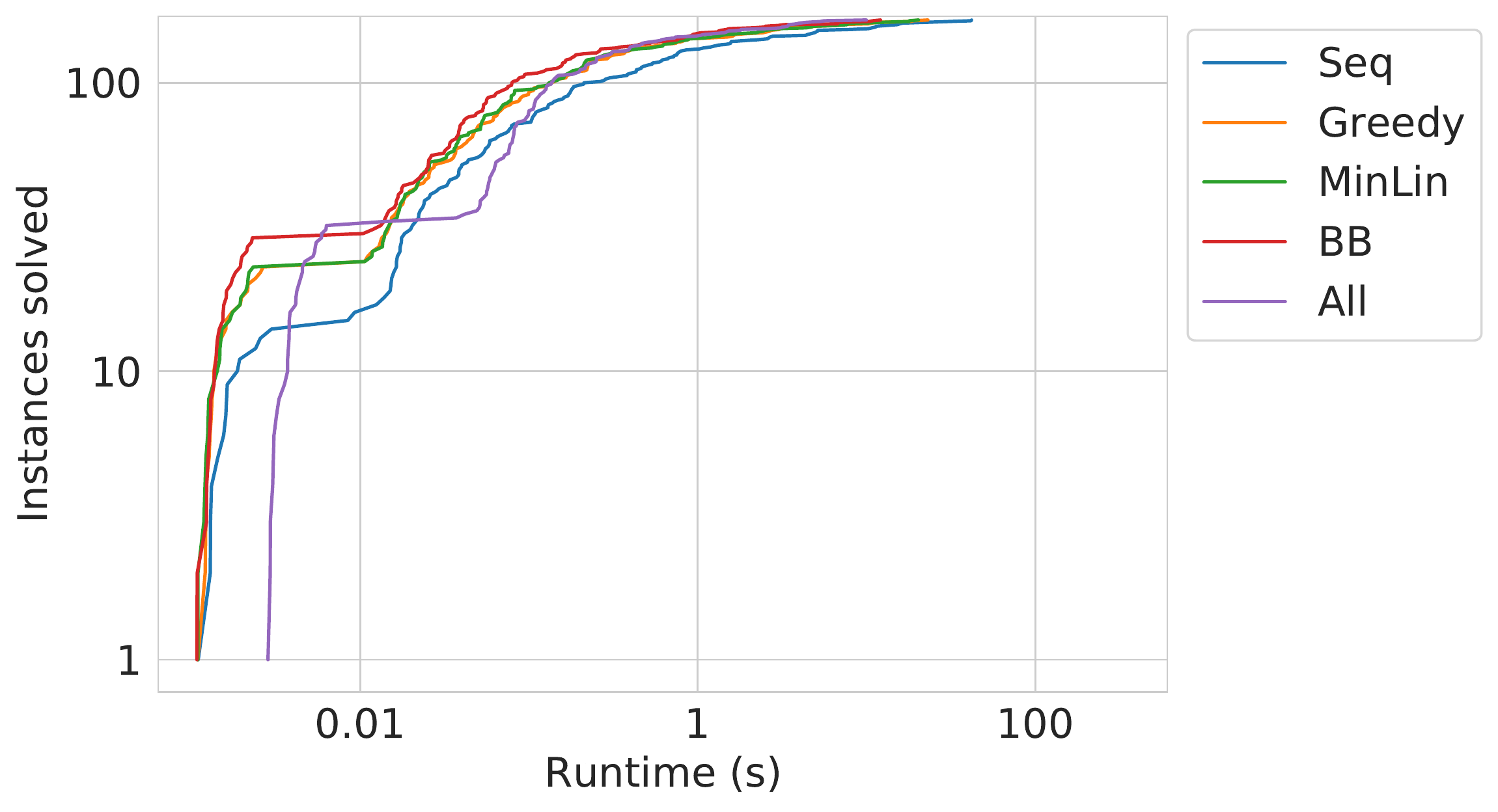}%
}
\subfloat[][\centering \texttt{mult4} \label{fig:qcp_d4}]{%
  \includegraphics[scale=0.35]{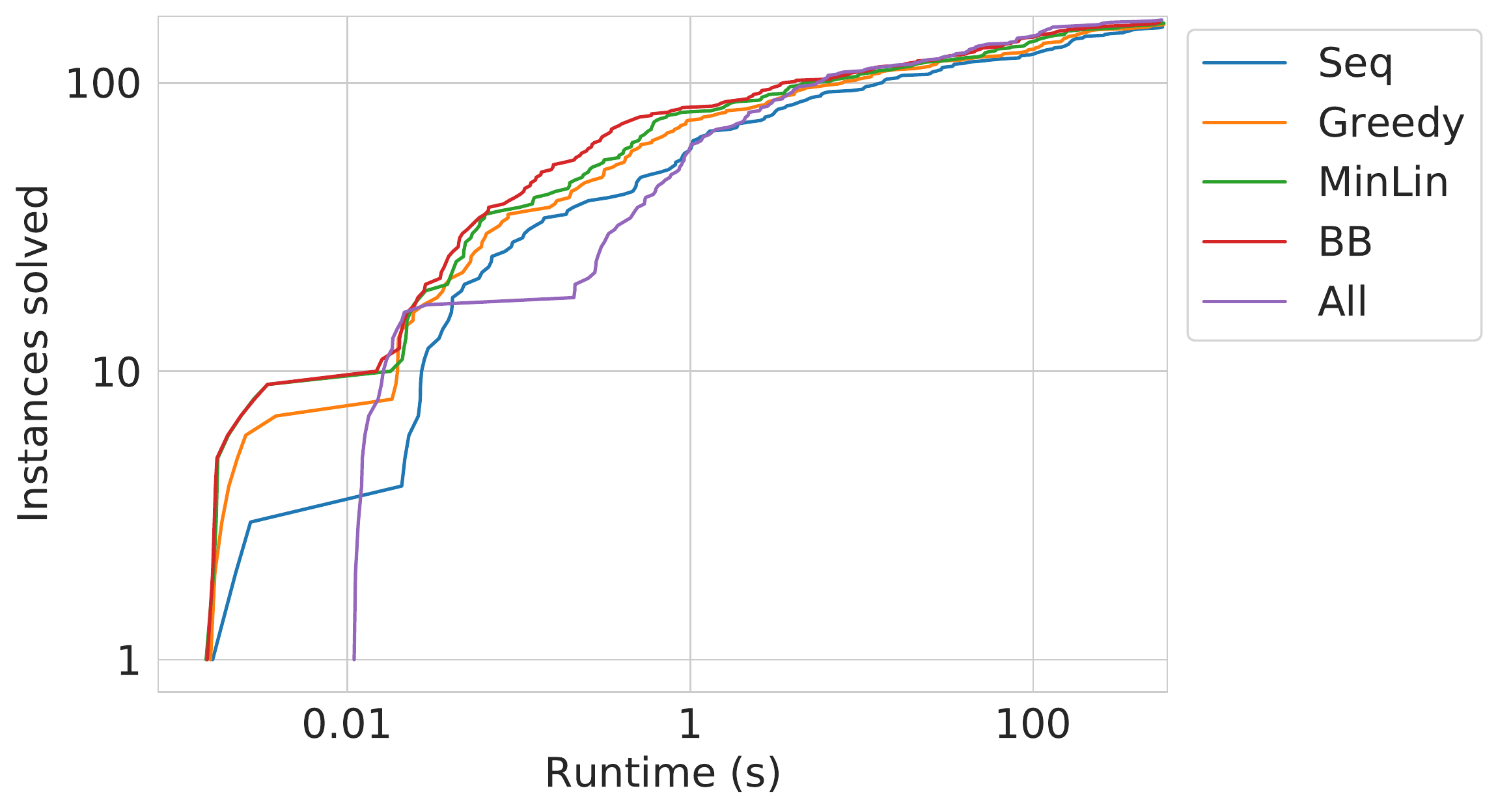}%
}
\caption{Comparison between the different linearizations when solving~\eqref{qcp}  (\texttt{mult3} and \texttt{mult4}).}
\label{fig:profiles_time_gap1}
\end{figure}
 
 Overall, the results  show that \texttt{BB} delivers solid performance and even beats \texttt{Full} for small runtimes. For \texttt{mult3} instances, all algorithms solve all instances to optimality, and \texttt{BB} has the best median solution time (0.05 seconds, versus 0.13 of \texttt{Seq}) and the second best average solution time (losing only to \texttt{All}). For \texttt{mult4},  \texttt{BB} closes the optimality gap for more instances than the other algorithms (except \texttt{All}) and has the smallest median runtime of~\texttt{BB} of 1.35 seconds (versus 3.6 seconds of \texttt{Seq}).

In contrast with \texttt{mult3} and \texttt{mult4}, the \texttt{vision} and \texttt{autocorr} instances are harder and could not be solved to optimality by any algorithm; the exceptions are two instances of \texttt{autocorr}, which are solved in a negligible amount of time by all algorithms. Therefore, we only report the optimality gaps for these data sets, using linear scale on both coordinates. 
\begin{figure}[htp]
\centering
\subfloat[][\centering \texttt{vision} \label{fig:qcp_vision}]{%
  \includegraphics[scale=0.35]{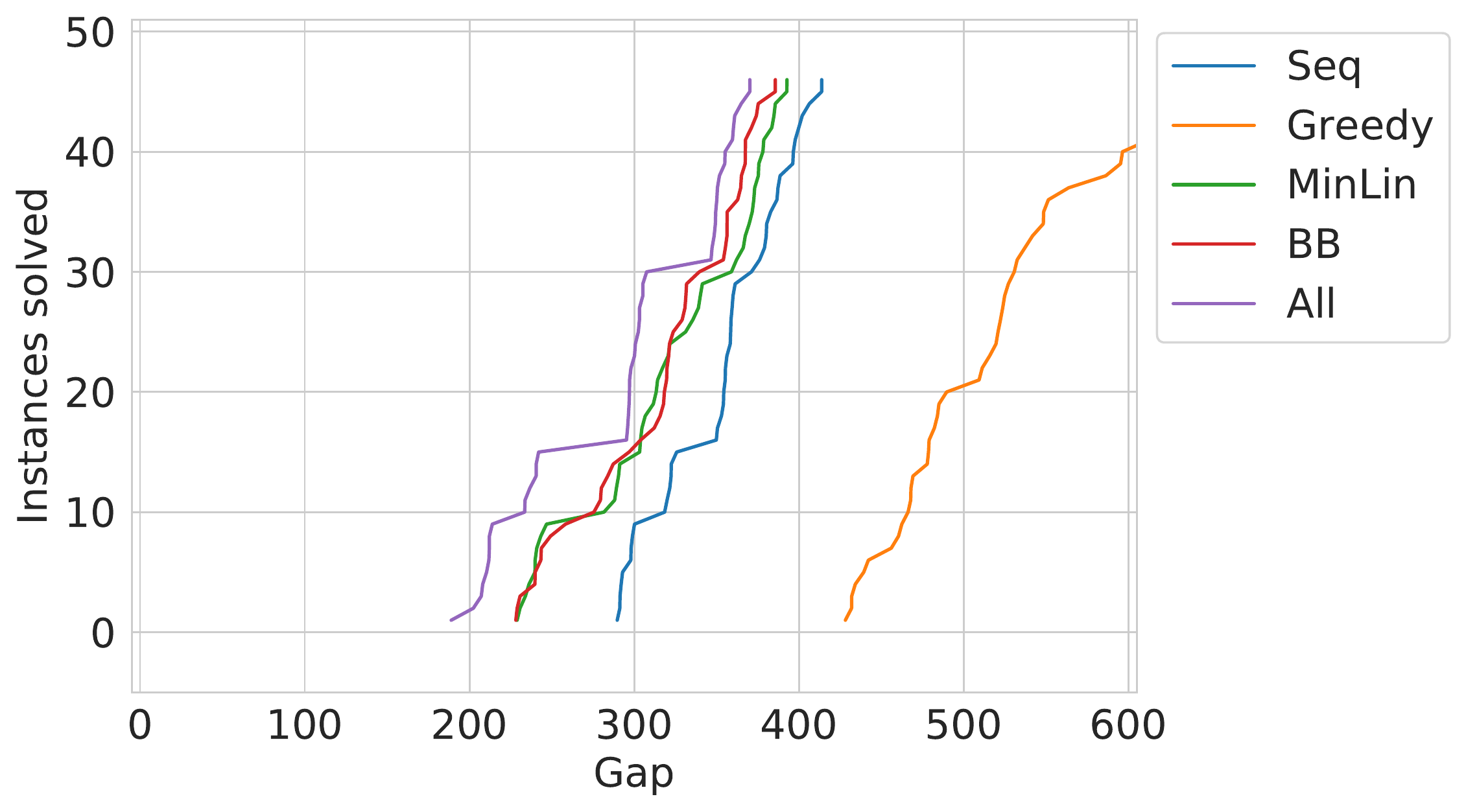}%
}
\subfloat[][\centering \texttt{autocorr}\label{fig:qcp_autocorr}]{%
  \includegraphics[scale=0.35]{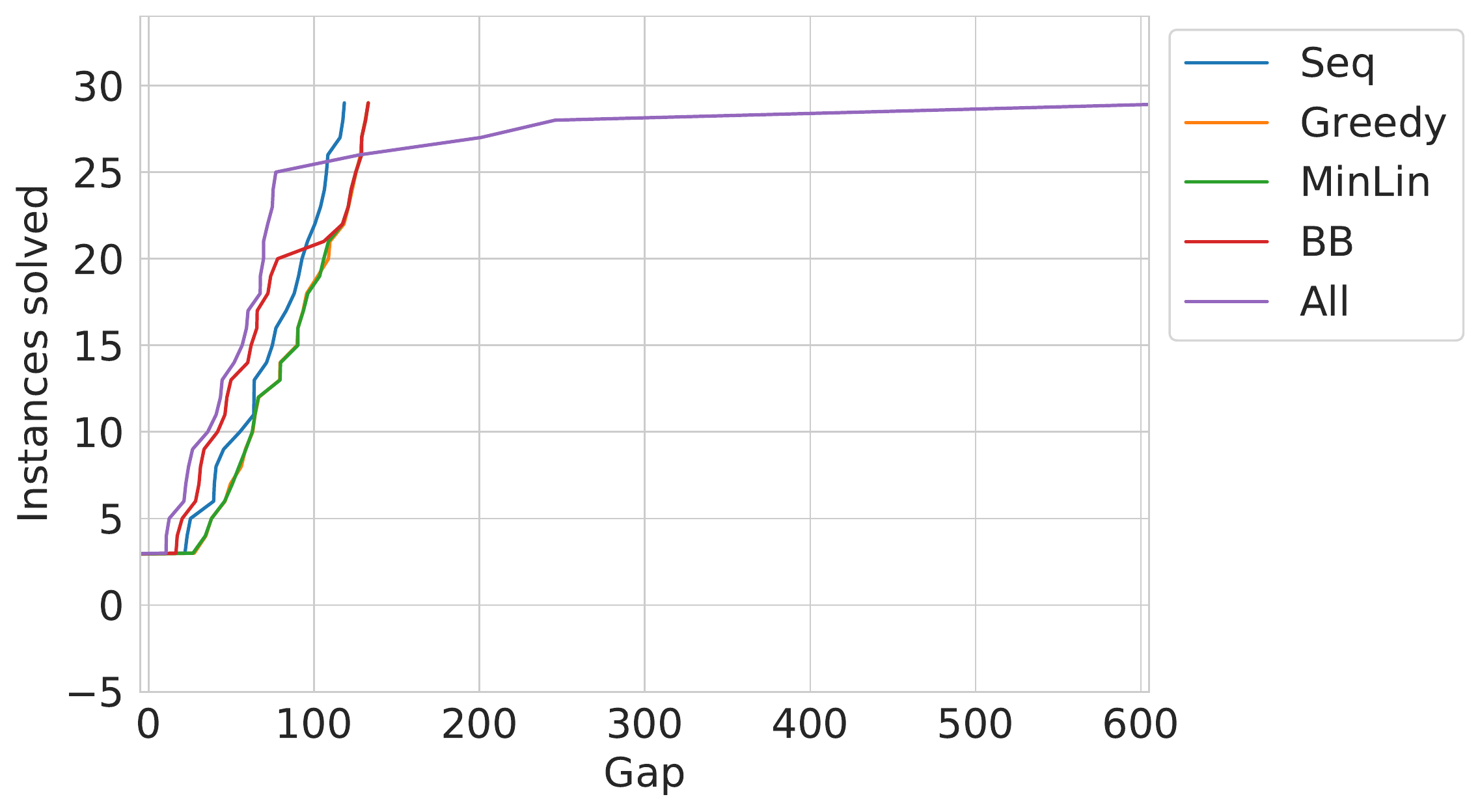}%
}
\caption{Comparison between the different linearizations when solving \eqref{qcp} (\texttt{vision} and \texttt{autocorr}).}
\label{fig:profiles_time_gap2}
\end{figure}

The results show that \texttt{All} is overall better; the result is not surprising, as a model with all triples is expected to be tighter. In contrast, \texttt{Greedy} is notably worse than all the other algorithms on the \texttt{vision} instances (see Example~\ref{ex:vision_greedy}). Both \texttt{BB} and \texttt{MinLin} consistently outperform \texttt{Seq} on the \texttt{vision} instances; in contrast, their performance is relatively similar for the~\texttt{autocorr} instances.

\section{Conclusion and Future Work}\label{sec:Conclusion}

In this work, we present a systematic investigation of linearization techniques of multilinear programs based on Recursive McCormick Relaxations. More precisely, we design algorithms to identify optimal linearizations  using two criteria: number of linearization terms and strength of the LP relaxation bound. The identification of a minimum-size linearization is shown to be NP-hard, and a greedy approach to the problem can deliver arbitrarily bad results, so we present an exact algorithm. We explore structural properties of the problem to derive a MIP that identifies a linearization of bounded cardinality delivering the best relaxation bound. Our algorithms are computationally more expensive than the linearization techniques used by the state-of-the-art nonlinear optimization solvers, but our computational results show that the additional computational overhead is compensated by the strength of the resulting linearized model, resulting in faster overall computational time. 

Our results are restricted to unconstrained multilinear programs, with either continuous and binary variables. One can easily adapt our algorithms to solve instances with linear or multilinear constraints, but preliminary experiments suggest that the impact of our algorithms is not as notable as in the unconstrained case, especially for linearizations that minimize the relaxation bound. The investigation and extension of our ideas to constrained problems can lead to an interesting research direction.


\bibliographystyle{plainnat} 
\bibliography{bib} 

\newpage

 \appendix

 \section{Proofs of Section~\ref{sec:minimumlinearization}  }

 \begin{proof}{Proof of Proposition~\ref{prop:greedy_bad}:} The proof of this proposition relies on the structural results presented in Section~\ref{sec:complexity}.  Let~$B = (U,V,E)$ be bipartite graph such that~$|U| = k$ for some~$k \in \mathbb{N}$, and let~$V \coloneqq \bigcup\limits_{i = 1}^k V_i$ (i.e., $V$ is partitioned into subsets $V_1,V_2,\ldots,V_k$) whereby~$|V_i| = \lfloor \frac{k}{i} \rfloor$, $i \in [k]$. We construct~$E$ by assigning exactly~$i$ neighbors in~$U$ to
each vertex in~$V_i$, $i \in [k]$. Moreover, each vertex in~$U$ has at most one neighbor in~$V_i$, and we assign neighbors in~$U$ to vertices in~$V_i$ in a way that the maximum degree of any vertex in~$U$ is~$k-1$. We obtain an instance of the 3-\polyopt{}  by applying the same construction presented in the proof of Theorem~\ref{thm:nphard} to the graph~$B$.

The greedy algorithm proceeds by selecting, in each iteration, the vertex with the largest number of uncovered neighbors. By construction, all the~$\sum_{i = 1}^k |V_i|$ pairs associated with~$V$ are incorporated to the linearization by the greedy algorithm, so the size of the solution is~$|V| \approx  \sum_{i = 1}^k \lfloor \frac{k}{i} \rfloor = \Theta(k \ln k)$. In contrast, a (minimum) linearization for the same instances picks all the pairs associated with~$U$, which contains only~$k$ elements, i.e., the proper triple set identified by \texttt{Greedy} is asymptotically~$O(\ln k)$ times larger than a minimum-sized one.
\end{proof}

 \begin{proof}{Proof of Theorem~\ref{thm:reductionrules}:}
For~\ref{rule1}, observe that as~$\trinomial$ shares no variables with other triples in~$\trinomials$, all pairs in~$\pairs(\trinomial)$ can only cover~$\trinomial$. Therefore, any optimal solution of~$f(\bfx)$ has exactly one element of~$\pairs(\trinomial)$. After the exhaustive application of~\ref{rule1}, each~$\trinomial$ has at least one neighbor~$\pair$ of degree at least 2. As any optimal solution that uses a neighbor of~$\trinomial$ of degree 1 may be replaced for another solution of same cardinality (or smaller) by using a neighbor of~$\trinomial$ of degree 2 instead, it follows that we can remove all elements of~$\pairs$ of degree 1, i.e., we can apply~\ref{rule2}. The application of~\ref{rule1} and~\ref{rule2} may lead to configurations where an element~$\trinomial$ of~$\trinomials$ has only one neighbor in~$\pairs$. As any feasible solution must contain at least one element of~$\pairs(\trinomial)$ for each~$\trinomial$ in~$\trinomials$, we apply~\ref{rule3}. From the validity of the previous rules, it follows that there is at least one optimal solution that does not contain elements of~$\pairs$ without neighbors, so~\ref{rule4} is valid. Finally, \ref{rule5}
 follows from the fact that a vertex~$\trinomial$ of~$\trinomials$ cannot be covered by any element of~$\pairs$ that does not belong to the same connected component in~$G$.
\end{proof}

 \begin{proof}{Proof of Proposition~\ref{prop:structurereduced}:} \ref{property1}  follows from the fact that~$|\pairs(\trinomial)| = 3$ in~$G$ for any~$\trinomial$ in~$\trinomials^r$ and from~\ref{rule3}. \ref{property2} follows directly from~\ref{rule2}. For~\ref{property3}, observe that any pair of elements~$\pair_1,\pair_2$ in~$\pairs$ sharing the same neighbors must have 
 exactly one variable in common.
 Therefore, there are exactly three variables associated with~$\pair_1$ and~$\pair_2$, so it defines exactly one element of~$\trinomials$, i.e., $\trinomials$ cannot have two distinct elements that are simultaneously neighbors of both~$\pair_1$ and~$\pair_2$. Finally, \ref{property4}  follows directly from~\ref{property3}, as any vertex in~$\pairs^r \setminus \{\pair\}$ can cover at most one neighbor of~$\pair$. 
 \end{proof}

 \begin{proof}{Proof of Theorem~\ref{thm:nphard}:}
The result  follows from a reduction of the vertex cover problem. In the vertex cover problem, we are given a graph~$G = (V,E)$ and the goal is to identify a subset~$V'$ of~$V$ such that for each edge~$e = (u,v)$ in~$E$ we have~$u \in V'$ or~$v \in V'$ (or both). The vertex cover problem is part of Karp's list of NP-complete problems~(\cite{karp1972reducibility}), and the problem is known to be hard even in planar graphs of degree at most 3~(\cite{garey1979computers}). 

Let~$G = (V,E)$ be the graph associated with an arbitrary instance~$I$ of the vertex cover problem. We construct the reduced bipartite graph~$G^r = (\pairs^r,\trinomials^r,E^r)$ associated with an instance of the 3-\polyopt{} as follows. For each vertex~$v$ in~$V$, we have an element~$y   x_v$ in~$\pairs$, and for each edge~$e = (u,v)$ in~$E^r$ we have an element~$x_u    x_v    y$ in~$T$. Informally, each vertex in~$V$ is associated with a pair in~$\pairs^r$ and each edge in~$E^r$ is associated with a triple in~$\trinomials^r$. A complete construction would also require the inclusion of~$(x_u,x_v)$ in~$\pairs^r$ for each~$(u,v)$ in~$E^r$; however, it follows from~\ref{property2} that we do not need to include them in~$\pairs^r$, as there is at least one optimal solution of~$(\pairs^r,\trinomials^r)$ that does not use elements of~$\pairs^r$ of degree 1. Therefore, we build~$E^r$ as in~\S\ref{sec:dominatingset}, but taking into account the transformations in~\S\ref{sec:reductionrules}.  For an example, see Figure~\ref{fig:nphardnessexample}.

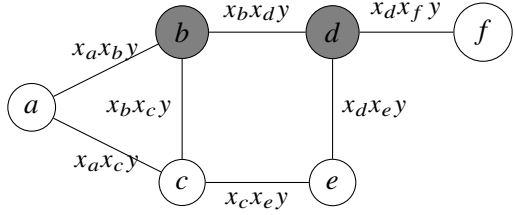
\begin{figure}
    \centering
\begin{tikzpicture}[main/.style = {draw, circle}] 

\node[main] (a) at (0,0) {$a$}; 
\node[main] (b)[fill=gray] at (2,1) {$b$};
\node[main] (c) at (2,-1) {$c$}; 
\node[main] (d)[fill=gray] at (4,1) {$d$};
\node[main] (e) at (4,-1) {$e$}; 
\node[main] (f) at (6,1) {$f$};

\path[every node/.style={font=\sffamily\small}]
        (a) edge node [above] {$x_a   x_b    y$} (b)
        (a) edge node [below] {$x_a    x_c    y$} (c)
        (b) edge node [left] {$x_b    x_c    y$} (c)
        (b) edge node [above] {$x_b    x_d    y$} (d)
        (c) edge node [below] {$x_c    x_e    y$} (e)
        (d) edge node [right] {$x_d    x_e    y$} (e)
        (d) edge node [above] {$x_d    x_f    y$} (f);

\end{tikzpicture} 
    
    \caption{Example of instance of the vertex cover problem for~$G = (V,E)$, where~$V = \{a,b,c,d,e\}$ and~$E = \{(a,b), (a,c), (b,c), (b,d), (c,e), (d,e), (d,f)\}$. The figure shows the monomials associated with each edge; namely, we have~$\trinomials = \{ x_a   x_b    y, x_a   x_c   y, x_b   x_c   y, x_b   x_d   y, x_c   x_e   y, x_d   x_e   y, x_d   x_f   y\}$.  The construction of the 3-\polyopt{} instance is complete with~$\pairs = \{x_a   y, x_b   y, x_c   y, x_d   y, x_e   y, x_f   y \}$. An optimal solution for the vertex cover instance is the set~$\{b,d\}$, whereas~$\{x_b   y, x_d   y \}$ is the optimal solution for the associated 3-\polyopt{} instance.}
    \label{fig:nphardnessexample}
\end{figure}

Any optimal solution~$V'$ for~$I$ is associated with a set of elements~$\pairs'$ in~$\pairs^r$ that cover each element of~$\trinomials^r$. In particular, the one-to-one relationships between~$V$ and~$\pairs^r$ and between~$E$ and~$\trinomials^r$ naturally extends to the coverage of edges by vertices in~$G$ and the coverage of triples by pairs in~$(\pairs^r,\trinomials^r)$. Therefore, it follows that the 3-\polyopt{} is NP-hard.
\end{proof}

 \begin{proof}{Proof of Theorem~\ref{thm:fpt}:} The structural properties of the reduced problem allow us to show that the 3-\polyopt{} is fixed-parameter tractable in the sizer~$k$ of the linearization;  
we denote this parameterized decision problem as~$(G^r,k)$. First, we show the adaptation of the kernelization procedure proposed by~\citet{buss1993nondeterminism} for the vertex cover problem applies to the 3-\polyopt{}.

\begin{lemma}[Rule 6]\label{rule6} If~$G^r$ contains an element~$\pair$ in~$\pairs$ with degree greater than or equal to~$k+1$, remove~$\pair$ and its neighbors and solve~$(G^r - \pair,k-1)$.
\end{lemma}
\begin{proof} This result follows from~\ref{property4}. Namely, if~$\pair$ has degree greater than or equal to~$k+1$, than any solution for the 3-\polyopt{} that does not contain~$\pair$ must contain at least~$k+1$ elements of~$\pairs \setminus \{\pair\}$ to cover its neighborhood. Similarly, any certificate showing that~$(G^r - \pair,k-1)$ is an ``yes'' instance can be efficiently converted in a ``yes'' certificate for~$(G^r,k)$.
\end{proof}

The deletion of a vertex~$\pair$ may affect all the elements in~$\trinomials$ as well as the elements in~$\pairs$, so the application of Rule 6 takes time~$O(|\pairs|(|\pairs|+|\trinomials|))$. Our fixed-parameter tractable procedure to solve an instance~$(G^r,k)$ of the 3-\polyopt{} consists of the application of Rules 1, 2, 3, 4, and 6; observe that, in addition to Rule 6, Rules 1 and 3 may also change (decrease) the value of~$k$. We can omit Rule 5 for the decision version of the problem.

 First, we claim that if~$(G^r,k)$ is a ``yes'' instance, then $|E| \leq k^2$. If Rule 6 (Proposition~\ref{rule6}) cannot be applied, all vertices in~$\pairs$ have at most~$k$ neighbors in~$\trinomials$. As at most~$k$ vertices of~$\pairs$ may be selected and, consequently, at most~$k^2$ vertices of~$\trinomials$ can be covered, it follows that~$|E| \leq k^2$. 

Next, we claim that if~$(G^r,k)$ is a ``yes'' instance, then $|\trinomials| \leq k^2/2$ and $|\pairs| \leq k^2/2$. From~\ref{property1}, each element of~$\trinomials$ must have at least 2 neighbors in~$\pairs$, so~$|\trinomials| \leq k^2/2$. Similarly, as~\ref{property2} shows that each element of~$\pairs$ has at least 2 neighbors in~$\trinomials$, it follows that~$|\pairs| \leq k^2/2$.

The exhaustive application of Rules 1, 2, 3, 4, and 6 can be performed in polynomial time. Namely, in each step, at least one vertex is removed, so in the worst case we have~$O( (|\pairs|+|\trinomials|)(2|\pairs| + |\trinomials| + |\trinomials|^2) + |\pairs|^2 + |\pairs||\trinomials|) = O( (|\pairs|+|\trinomials|)(|\pairs|^2+|\trinomials|^2)) = O(w^3)$, where~$w = |\pairs| + |\trinomials|$ represents the size of the instance. A bounded search tree on the kernel needs time~$T(w,k) = O(3^k n)$; each vertex in~$\trinomials$ has at most 3 neighbors, and a vertex of~$\pairs$ can be removed (with its neighbors in~$\trinomials$) in time~$O(w)$. As~$w = O(k^2)$ after the kernelization procedure, the brute-force procedure consumes time~$O(3^k   k^2)$. In total, the algorithm consumes time~$O(w^3 + 3^k   k^2)$ = $O(k^6 + 3^k   k^2)$, and therefore the 3-\polyopt{} is fixed-parameter tractable. 
\end{proof}

\end{document}